\documentclass[12pt, reqno]{article}

\usepackage{fullpage}
\usepackage{amssymb, amsmath, amsthm}
\usepackage{enumerate}
\usepackage{enumitem}
\usepackage[colorlinks,breaklinks, allcolors=blue]{hyperref}
\usepackage{xcolor}
\usepackage{longtable}
\usepackage{cleveref}
\usepackage{tikz-cd}
\usepackage{caption}
\usepackage{xypic}
\usepackage{url}

\numberwithin{equation}{section}

\newcommand{\U}{\mathrm{U}}
\newcommand{\SU}{\mathrm{SU}}

\newcommand{\su}{\mathfrak{su}}

\newcommand{\mft}{\mathfrak{t}}
\newcommand{\mfg}{\mathfrak{g}}

\newcommand{\Ad}{\mathrm{Ad}}

\newcommand{\RR}{\mathbb{R}}
\newcommand{\CC}{\mathbb{C}}
\newcommand{\ZZ}{\mathbb{Z}}
\newcommand{\PP}{\mathbb{P}}
\newcommand{\Id}{\mathrm{Id}}

\DeclareMathOperator{\Aut}{Aut}

\DeclareMathOperator{\diag}{diag}

\newcommand{\GL}{\mathrm{GL}}

\newtheorem{thm}{Theorem}[section]

\newtheorem{theorem}[thm]{Theorem}

\newtheorem{lem}[thm]{Lemma}
\newtheorem{lemma}[thm]{Lemma}

\theoremstyle{definition}
\newtheorem{rem}[thm]{Remark}
\newtheorem{defn}[thm]{Definition}
\newtheorem{ex}[thm]{Example}
\newtheorem{remark}[thm]{Remark}
\newtheorem{definition}[thm]{Definition}
\newtheorem{example}[thm]{Example}

\newcommand{\C}{\mathbb C}
\newcommand{\Z}{\mathbb Z}

\newcommand{\Q}{\mathbb Q}
\newcommand{\R}{\mathbb R}

\begin{document}
	\title{On GKM fiber bundles and realizability with full flag fibers}
	\author{Oliver Goertsches\footnote{Philipps-Universit\"at Marburg, email:
			goertsch@mathematik.uni-marburg.de}, Panagiotis Konstantis\footnote{Philipps-Universit\"at Marburg,
			email: pako@mathematik.uni-marburg.de },\\[0.1cm]
		Nikolas Wardenski\footnote{University of Haifa,
			email: wardenski.math@gmail.com}, and Leopold
		Zoller\footnote{Universit\"at zu K\"oln, email: zoller@math.uni-koeln.de}}
	
	\maketitle
	\begin{abstract}
		We investigate under which conditions an equivariant fiber bundle whose base, total space and fiber are GKM manifolds induces a fibration or fiber bundle of the corresponding GKM graphs. In particular, we give several counterexamples. Concerning the converse direction, i.e., the realization problem for fiber bundles of GKM graphs, we restrict to the setting of fiberwise signed GKM fiber bundles over $n$-gons whose fiber is the GKM graph of a full flag manifold. While it was known that any such bundle is realizable for a $\CC P^1$-fiber, we observe that new phenomena occur in higher dimensions where realizability depends on the twist automorphism of the GKM fiber bundle. We classify possible twist isomorphsims and show that realizability can be decided in terms of our classification.
	\end{abstract}
	
\section{Introduction}	
	
One of the most striking connections between combinatorics and the geometry of manifolds is the
Delzant correspondence between symplectic toric manifolds and Delzant polytopes \cite{MR984900}.
Delzant's Theorem states that both worlds are equivalent. While the
spirit of this phenomenon carries
over to more general settings, the precise nature of the correspondence between geometry
and combinatorics is more obscure outside of the toric context. In this paper we are
concerned with the notion of GKM manifolds, named after and popularized by
\cite{MR1489894}. In this setting, much of the topology of the manifold is encoded in a
labeled graph, the GKM graph. The abstract notion of a GKM graph has been introduced in
\cite{MR1823050} and studied independently from geometry. Many classical geometric
phenomena turn out to have combinatoric counterparts. In particular, the authors introduce
the notion of a fibration of GKM graphs. This is further developed in \cite{MR2873096}, where the
refined notion of a fiber bundle of GKM graphs is introduced and a combinatorial
Leray-Hirsch Theorem is proved.

In light of the original motivation of understanding the relation between both worlds,
an immediate natural question is the following: which abstract GKM graphs are realized
geometrically by a GKM manifold? While this is largely unknown, the situation has been
successfully studied in dimensions $\leq 6$ \cite{2210.01856v1} where a one-to-one correspondence
between certain GKM manifolds and graphs continues to hold akin to the toric situation. The
problem can of course be asked subject to more specialized situations, and the goal of this
paper is to shed light on the correspondence between the geometric and combinatoric
notions of fiber bundles. More specifically, two immediate questions arise:
\begin{enumerate}[label = $(\roman*)$]
\item When does an equivariant fiber bundle of GKM manifolds induce the combinatoric situation of a fiber bundle of GKM graphs?
\item When is a fiber bundle of GKM graphs realized by an equivariant fiber bundle of GKM manifolds?
\end{enumerate}
These questions have been addressed in \cite{MR4634089} specifically for manifolds of dimension $6$. Somewhat simplified, the answer to both questions turns out to be always positive, hence allowing free passage between geometry and combinatorics. The main results of this paper illustrate that neither question has a simple answer in higher dimensions.

We study question $(ii)$ in the situation where the fiber graph $\Gamma$ is that of a
generalized flag $G/T$, where $G$ is a simply connected compact Lie group with maximal
torus $T\subset G$, and the base graph $B$ is $2$-regular. In this situation, the
combinatorics of the fiber bundle are governed by a single holonomy automorphism $\Phi$ of
$\Gamma$ that arises from a path around the base graph. We study the automorphism group
of $\Gamma$ and obtain a structural result in Theorem \ref{thm:type decomposition}, showing
that, in case $\Phi$ preserves the signed structure on on $\Gamma$ induced by a
$T$-invariant Kähler structure, there is a unique decomposition $\Phi = \Phi_1\circ
\Phi_2$ with $\Phi_1$ coming from left multiplication with an element of the Weyl group of
$G$ (Type 1) and $\Phi_2$ induced by an outer automorphism of $G$ (Type 2). The answer to
question $(ii)$ above then can be completely answered in terms of this decomposition (cf.
Theorem \ref{thm:realization}):

\begin{theorem}\label{thm: 1introduction}
		Let $\Gamma\to \Gamma'\to B$ be a $T$-GKM fiber bundle, where
		$\Gamma$ is the GKM graph of $G/T$ and $B$ is $2$-regular. Furthermore assume that the common kernel of the weights of $B$ is a connected subgroup of $T$ and that the twist automorphism $\Phi=\Phi_1\circ \Phi_2$ preserves the connection and the signed structure on $\Gamma$, where $\Phi_1, \Phi_2$ denotes the unique decomposition in Type 1 and Type 2 automorphisms.
		\begin{enumerate}[label = $(\roman*)$]
			\item If ${\Phi}_2=\Id$, then the bundle is realizable. More precisely, there exists a smooth $T$-equivariant fiber bundle $Z\to X$ with fibers over fixed points twisted equivariantly diffeomorphic to $G/T$, such that the $T$-action on $Z$ is of GKM type with GKM graph $\Gamma'$ and $X$ is a quasitoric $4$-fold with GKM graph $B$.
			\item If ${\Phi}_2\neq \Id$, then the map $H^*(\Gamma')\rightarrow H^*(\Gamma)$ induced by the fiber inclusion is not surjective. In particular, the GKM fiber bundle is not realizable by an equivariant fiber bundle in which base, total space, and fibers over fixed points are GKM manifolds.
		\end{enumerate}	
		
	\end{theorem}
	
In particular, unlike in the low-dimensional situation, the answer can be both positive and negative. Concrete examples of the realizable and the non-realizable case are provided in Examples \ref{ex:fiberbundles} and \ref{ex:Type2} respectively.

With regards to question $(i)$, the answer turns out to be that it is not easy to give conditions which are simultaneously sufficient and necessary to ensure a fiber bundle of GKM manifolds induces a fiber bundle of GKM graphs (cf. Remark \ref{rem: pessimistic rem}). Regarding sufficient conditions, we have (cf. Theorem \ref{thm: induced graph bundles})
\begin{theorem}\label{thm: 2introduction}
Let $F\rightarrow M\rightarrow B$ be a fiber bundle in which all spaces are $T$-GKM manifolds and the projection map $M\to B$ is equivariant.
\begin{enumerate}[label = $(\roman*)$]
\item Then the induced maps on one-skeleta induce a GKM fibration.
\item If additionally, for each vertex $v$ of the graph of $M$, the weights of any horizontal edge adjacent to $v$ and any two vertical edges adjacent to $v$ are linearly independent, then transport along a horizontal edge induces a graph isomorphism between the respective fiber graphs.
\item Assume that for each vertex $v$ of the graph of $M$, the weights of any horizontal edge adjacent to $v$ and any three vertical edges adjacent to $v$ are linearly independent. Furthermore assume that any adjacent set of vertical edges can be completed to a basis of the weight lattice. Then the induced GKM fibration is a GKM fiber bundle.
\end{enumerate}
\end{theorem}

Part $(i)$ was proved prior to this paper in \cite{MR4634089} but we repeat the proof here since it
is more or less contained in the considerations needed for $(ii)$ and $(iii)$. The jump
from a GKM fibration to a GKM fiber bundle requires horizontal transport to respect the
graph structures of the fibers --as in $(ii)$-- and requires the labels of the fiber
graphs to transform uniformly through a linear automorphism --which is assured by the
additional conditions of $(iii)$. However both properties can individually fail for an
equvariant fiber bundle of GKM manifolds, even under strong additional assumptions on the
equivariant local trivializations (see Examples \ref{ex: cool counterexample} and \ref{ex:
long counterexample}). In particular, this shows that additional conditions as in $(ii)$
and $(iii)$ are needed in order for the statement to hold. On the other hand, the
combinatoric restrictions made in $(ii)$ and $(iii)$ are pretty severe and far from necessary:
they are not satisfied by most of the examples of GKM fiber bundles which we have shown to
be realizable in Theorem \ref{thm: 1introduction} part $(i)$.

The counterexamples \ref{ex: cool counterexample} and \ref{ex: long counterexample} might
suggest that the notion of a GKM fiber bundle is too restrictive for the general geometric
setup considered in this paper. However, at the same time there are very natural examples
of fiber bundles of GKM graphs arising from geometry \cite[Theorem 4.1]{MR2873096} and the conditions
are not restrictive enough to ensure geometric realizability as shown by Theorem \ref{thm:
1introduction} part $(ii)$. In conclusion, the correspondence between geometry and
combinatorics is partially successful when restricted to fiber bundles of GKM manifolds,
but eludes a complete and simple description in the currently developed terminology.

The structure of the article is as follows: In Section  \ref{sec:preliminaries} we recap
some important notions about GKM manifolds, GKM fibrations and fiber bundles, as well as
some results about GKM actions on homogeneous spaces and quasitoric manifolds. Section
\ref{sec:equivariant fiber bundles} investigates in more detail the conditions under which
an equivariant fiber bundle, whose base, total space, and fiber are GKM manifolds, induces
a GKM fibration or a GKM fiber bundle. Section \ref{sec:autos} classifies
automorphisms of GKM graphs of full flag manifolds, which are used in Section
\ref{sec:Realizability} to understand realizability of GKM fiber
bundles whose fibers are GKM graphs of full flag manifolds. Finally, in Section \ref{sec:
examples} we construct examples of GKM fiber bundles that can be realized in the sense
of Section \ref{sec:Realizability}.

\paragraph{Acknowledgements:}
 The first, second and fourth named author gratefully acknowledge funding of the Deutsche
 Forschungsgemeinschaft (DFG, German Research Foundation) -- 452427095.
 \section{Preliminaries}\label{sec:preliminaries}

\subsection{GKM actions and graphs}

For an action of a compact torus $T$ on a smooth manifold $M$, we define its
\emph{$k$-skeleton} $M_k$ to be 
\[
M_k:=\{p\in M\mid \dim T\cdot p\leq k\},
\]
which is the union of all $T$-orbits of dimension at most $k$. In GKM theory, named after
Goresky--Kottwitz--MacPherson \cite{MR1489894}, of particular importance is the $0$-skeleton,
which is the same as the $T$-fixed point set $M^T$, and the $1$-skeleton of the action. 
\begin{defn}
    Assuming $M$ to be compact, connected and orientable, with vanishing odd-dimensional rational cohomology, we say that the $T$-action is \emph{GKM}, or \emph{of GKM type}, if $M^T$ is a finite set of points and $M_1$ a finite union of $T$-invariant $2$-spheres.
\end{defn}

The definition is tailored towards the fact that for a GKM $T$-action on $M$, the orbit space $M_1/T$ has the structure of a graph (with vertices corresponding to fixed points, and one edge for each invariant $2$-sphere). This graph carries a natural labeling of the edges by weights of the isotropy representations. More precisely, for any fixed point $p\in M^T$, the $T$-isotropy representation decomposes into $2$-dimensional irreducible submodules, each one being tangent to exactly one of the invariant $2$-spheres containing $p$. We label the corresponding edge with the weight of this submodule, which is an element in $\ZZ_\mft^*/\pm 1$, where $\mft$ is the Lie algebra of $T$ and $\ZZ_\mft^*\subset \mft^*$ is the weight lattice of $T$. This labelled graph is called the \emph{GKM graph} of the $T$-action.

In this paper we will be concerned with the realization problem for abstract GKM graphs.
Let us recall this notion, going back to \cite{MR1823050}. We consider graphs $\Gamma$ with finite vertex set $V(\Gamma)$ and finite edge set $E(\Gamma)$; we allow multiple edges between vertices, but no loops. Edges of graphs do not have a fixed orientation; formally, $E(\Gamma)$ contains every edge twice, once for each possible orientation. In this way, for an edge $e\in E(\Gamma)$, it is well-defined to speak about its initial vertex $i(e)$ and terminal vertex $t(e)$. For $e\in E(\Gamma)$ we denote by $\bar{e}$ the same edge, but with opposite orientation. For a vertex $v$ we denote $E(\Gamma)_v=\{e\in E(\Gamma)\mid i(e) = v\}$.
\begin{defn}
    A \emph{connection} on an $n$-valent graph $\Gamma$ is a collection of bijections $\nabla_e:E(\Gamma)_{i(e)}\to E(\Gamma)_{t(e)}$, for $e\in E(\Gamma)$, satisfying
		\begin{enumerate}[label=$(\roman*)$]
        \item $\nabla_ee=\bar{e}$ and
        \item $\nabla_{\bar{e}} = (\nabla_e)^{-1}$ for any $e\in E(\Gamma)$.
		\end{enumerate}
\end{defn}
In the following, for an element $a\in \ZZ^m/\pm 1$ we will call a \emph{lift} of $a$ any of the two elements in $\ZZ^m$ that project to $a$. Moreover, note that the notion of linear independence is meaningful for elements in $\ZZ^m/\pm 1$.
\begin{defn}\label{defn:abstractgkmgraph}
    We fix natural numbers $n$ and $m$. An \emph{(abstract) GKM graph} $(\Gamma,\alpha)$ is a pair of an $n$-valent graph $\Gamma$ and a labeling of the edges $\alpha:E(\Gamma)\to \ZZ^m/\pm 1$, called \emph{axial function}, satisfying the following properties:
    \begin{enumerate}[label=$(\roman*)$]
        \item For any $v\in V(\Gamma)$ and $e,f\in E(\Gamma)_v$, we have that $\alpha(e)$ and $\alpha(f)$ are linearly independent.
        \item There exists a connection $\nabla$ on $\Gamma$ which is \emph{compatible with $\alpha$}: for any $v\in V(\Gamma)$ and $e,f\in E(\Gamma)_v$ and any choice of lifts $\tilde{\alpha}(e)$ of $\alpha(e)$ and $\tilde{\alpha}(f)$ of $\alpha(f)$ there is $\varepsilon\in \{\pm 1\}$ and $c\in \ZZ$ such that
        \[
        \tilde{\alpha}(\nabla_e f) = \varepsilon \tilde{\alpha}(f) + c\tilde{\alpha}(e).
        \]
        \item $\alpha(\bar{e})=\alpha(e)$ for all $e\in E(\Gamma)$.
    \end{enumerate}
If, for any $v\in V(\Gamma)$ and any $k$ adjacent edges $e_1,\ldots,e_k$,
the labels $\alpha(e_1),\ldots,\alpha(e_k)$ are linearly independent, we call $\Gamma$ an \emph{(abstract)
GKM$_k$ graph}. 
\end{defn}
The GKM graph of a GKM action is always an abstract GKM graph in the sense of Definition
\ref{defn:abstractgkmgraph}. The existence of a compatible connection was shown in
\cite{MR1823050} and \cite[Proposition 2.3]{MR4363804}.

\begin{defn}
A \emph{signed (abstract) GKM graph} is a pair $(\Gamma,\alpha)$ of an $n$-valent graph $\Gamma$ and an \emph{axial function} $\alpha:E(\Gamma)\to \ZZ^m$ such that (i) and variants of (ii) and (iii) in Definition \ref{defn:abstractgkmgraph} hold: in (ii), there shall exist for all $v$ and $e,f\in E(\Gamma)_v$ some $c\in \ZZ$ such that $\alpha(\nabla_ef) = \alpha(f) + c\alpha(e)$; in (iii), we demand $\alpha(\bar{e}) = -\alpha(e)$ for all edges.

Given a signed abstract GKM graph $(\Gamma,\alpha)$, the composition $\pi\circ \alpha$, where $\pi:\ZZ^m\to \ZZ^m/\pm 1$ is the projection, defines on $\Gamma$ an axial function in the sense of Definition \ref{defn:abstractgkmgraph}. We say that $(\Gamma,\alpha)$ is a \emph{signed structure} on the abstract GKM graph $(\Gamma,\pi\circ \alpha)$.
\end{defn}

In Section \ref{sec:GKMfiberbundles} below, we will need a rather general notion of morphism between graphs which allows for the collapsing of edges.
\begin{defn}\label{defn:morphismgraphs}
Given graphs $\Gamma$ and $\Gamma'$ as before (finite, possibly with multiple edges, but without loops), a \emph{morphism} $\Phi:\Gamma\to \Gamma'$ consists of a map $V(\Gamma)\to V(\Gamma')$ on vertices, as well as a partial map on edges, both again denoted by $\Phi$. More precisely, any edge $e\in V(\Gamma)$ with $\Phi(i(e))\neq \Phi(t(e))$ is sent to an edge $\Phi(e)$ connecting $\Phi(i(e))$ with $\Phi(t(e))$. We assume that $\Phi(\bar{e})=\overline{\Phi(e)}$; on edges $e$ with $\Phi(i(e))=\Phi(t(e))$ the map $\Phi$ is not defined.
\end{defn}
In case of a morphism $\Phi:\Gamma\to \Gamma'$ such that $\Phi:V(\Gamma)\to V(\Gamma')$ and $\Phi:E(\Gamma)\to E(\Gamma')$ are bijective, this definition restricts to the usual definition of an isomorphism of graphs.

\begin{defn}\label{defn:gkmgraphiso}
		Consider two abstract GKM graphs $(\Gamma,\alpha)$ and $(\Gamma',\alpha')$, where both
		axial functions $\alpha$ and $\alpha'$ take values in $\ZZ^m/\pm 1$. Then an
		isomorphism $\Phi:\Gamma\to \Gamma'$ is an \emph{isomorphism of GKM graphs}
		together with a linear isomorphism $\Psi:\ZZ^m\to \ZZ^m$ such that 
    \[
        \alpha'(\Phi(e)) = \Psi(\alpha(e))
    \]
		for all $e\in E(\Gamma)$, where $\Psi$ denotes also the
		induced map $\mathbb{Z}^{m}/\pm 1  \to \mathbb{Z}^{m}/\pm 1$. If $\Gamma=\Gamma'$, then we call $\Phi$ an
		\emph{automorphism of the GKM graph} $\Gamma$.
\end{defn}

One has natural isomorphisms $\mft^*_\ZZ \cong H^1(T;\ZZ)\cong H^2(BT;\ZZ)$ where $BT$ denotes the classifying space of $T$. Hence the labels embed into the ring $H^*(BT;\ZZ)$ which is a polynomial ring generated in degree $2$. For a GKM graph $(\Gamma,\alpha)$ one defines the (equivariant) graph cohomology
\[H^*_T(\Gamma) = \left\{f\in \prod_{v\in V(\Gamma)} H^*(BT;\ZZ) ~|~ f_{i(e)}\equiv f_{t(e)}\mod \alpha(e) \text{ for all } e\in E(\Gamma)\right\} \]
and the non-equivariant graph cohomology is given by
\[H^*(\Gamma) = H^*_T(\Gamma)/ (H^{>0}(BT;\ZZ)\cdot H^*_T(\Gamma)).\]
Then by \cite{MR1489894}, if $\Gamma$ is the GKM graph of a GKM manifold $M$, one indeed has $H^*_T(M;\ZZ)\cong H_T(\Gamma)$ and $H^*(M;\ZZ)\cong H(\Gamma)$. Although other coefficient rings are possible we will always use integral coefficients in this paper and suppress coefficients from the notation.

An isomorphism of GKM graphs induces a map on the graph cohomology.
If $\Phi\colon \Gamma\rightarrow \Gamma'$ is an isomorphism with compatible linear transformation $\Psi\colon \ZZ^m\rightarrow\ZZ^m$, and $f\in H_T(\Gamma')$ then one defines the pullback $\Phi^*\colon H_T^*(\Gamma')\rightarrow H_T^*(\Gamma)$ via $\Phi^*(f)_v =\Psi^{-1}(f_{\Phi(v)})$. One quickly checks that the divisibility conditions in the definition of $H_T^*(\Gamma)$ are satisfied for $\Phi^*(f)$. By taking quotients one obtains an induced map $H^*(\Gamma')\rightarrow H^*(\Gamma)$.

\subsection{GKM fiber bundles}\label{sec:GKMfiberbundles}

Let us review the notions of GKM fibration and GKM fiber bundle, introduced in
\cite{MR2873096}, slightly modified as in \cite{MR4634089} to take
account of multiple edges between vertices, see Remark \ref{rem: scheissegal}.

Given a morphism $\pi:\Gamma'\to B$ of graphs as in Definition \ref{defn:morphismgraphs}, we call an $e\in E(\Gamma')$ \emph{vertical} if $\pi(i(e))=\pi(t(e))$, otherwise \emph{horizontal}. For $v\in V(\Gamma')$, we denote by $H_v\subset E(\Gamma')_v$ the set of horizontal edges emanating from $v$.

\begin{defn}
    A morphism $\pi:\Gamma'\to B$ is called a \emph{graph fibration} if for all $v\in V(\Gamma')$, $\pi$ defines a bijection $H_v\to E(B)_{\pi(v)}$.
\end{defn}
Given a graph fibration $\pi:\Gamma'\to B$ and a vertex $v\in V(\Gamma')$, an edge $e\in E(B)_{\pi(v)}$ admits a unique horizontal lift $\tilde{e}$ to an edge at $v$.

\begin{defn}\label{defn:GKMfibration}
    Let $(\Gamma',\alpha)$ and $(B,\alpha_B)$ be abstract GKM graphs,
		such that both axial functions map to $\mathbb{Z}^{m}/\pm 1$. Then a graph fibration $\pi:\Gamma'\to B$ is called a \emph{GKM fibration} if there exist connections $\nabla$ on $\Gamma'$ and $\nabla^B$ on $B$, compatible with $\alpha$ respectively $\alpha_B$, such that
    \begin{enumerate}[label=$(\roman*)$]
        \item For every horizontal edge $e$ in $\Gamma'$ we have $\alpha^B(\pi(e)) = \alpha(e)$.
        \item For every edge $e\in E(\Gamma')$, the connection $\nabla_e:E(\Gamma')_{i(e)}\to E(\Gamma')_{t(e)}$ respects the decomposition into horizontal and vertical edges.
        \item For any $v\in V(\Gamma')$ and horizontal edges $e,e'\in E(\Gamma')_v$ we have $\pi(\nabla_e e') = \nabla^B_{\pi(e)} \pi(e')$.
    \end{enumerate}
\end{defn}

For a GKM fibration $\pi:(\Gamma',\alpha)\to (B,\alpha_B)$ and a vertex $p\in V(B)$, we define $\Gamma_p$ to be the subgraph of $\Gamma'$ with vertex set $V(\Gamma_p):=\pi^{-1}(p)$ and edge set $E(\Gamma_p)$ all the (vertical) edges in $\Gamma'$ connecting two of these vertices. If $\Gamma'$ is $n$-valent, and $B$ $m$-valent, then $\Gamma_p$ is an $(n-m)$-valent GKM graph.

 \begin{rem}\label{rem:horiztransportunique}
In general, given the connection $\nabla^B$, there could be several connections $\nabla$ on $\Gamma'$ satisfying the conditions of Definition \ref{defn:GKMfibration}: while the transport of horizontal edges is uniquely determined, the transport of vertical edges is not. However:  in case the weights of any horizontal edge and any two adjacent vertical edges are linearly independent, then for any horizontal edge $e\in E(\Gamma')$, the connection $\nabla_e$ is also uniquely determined on vertical edges; in case the weights of any three adjacent vertical edges are linearly independent, then for any vertical edge $e\in E(\Gamma')$, the connection $\nabla_e$ is uniquely determined on vertical edges.
\end{rem}

Any edge $e\in E(B)$, with $p:=i(e)$ and $q:=t(e)$, defines a bijection
$\Phi_e:V(\Gamma_p)\to V(\Gamma_q)$, by sending a vertex $v\in V(\Gamma_p)$ to the
endpoint of the horizontal lift $\tilde{e}$ of $e$ with $i(\tilde{e})=v$. In general, this
map is not part of a morphism of graphs $\Phi_e:\Gamma_p\to \Gamma_q$, see 
\cite[Example2.10]{MR2873096}. Even if it is, this graph morphism is not necessarily unique, as we allow
multiple edges between vertices. In the following definition, however,
there is such a graph morphism, which is even compatible with some choice of connections.

\begin{defn}\label{defn:GKMfiberbundle}
		Consider a GKM fibration $\pi:(\Gamma',\alpha)\to (B,\alpha_B)$, together with
		compatible connections $\nabla$ on $\Gamma'$ and $\nabla^B$ on $B$ as in Definition
		\ref{defn:GKMfibration}. We assume additionally that for every $v\in V(\Gamma')$,
		every horizontal edge $\tilde{e}\in H_v$, with $e=\pi(\tilde{e})$, and every vertical
		edge $e'\in E(\Gamma')_v$, the edge $\nabla_{\tilde{e}}e'$ connects $\Phi_e(v)$ with
		$\Phi_e(t(e'))$. In this case we define $\Phi_e(e'):=\nabla_{\tilde e}e'$
		and obtain
		isomorphisms of graphs $\Phi_e:\Gamma_p\to \Gamma_q$, for $p:=i(e)$ and $q:=t(e)$. In
		case all these maps are isomorphisms of GKM graphs, i.e., there exist compatible
		linear isomorphisms $\ZZ^m\to \ZZ^m$, we call $\pi$ a \emph{GKM fiber
		bundle}.
\end{defn}

 \begin{rem}\label{rem: scheissegal}  We have made slight adjustments to the notion of GKM
	 fiber bundle when compared to its introduction in \cite{MR2873096}. Most importantly we are considering integer coefficients while the original reference works over $\R$. The reason for this choice is that GKM graphs of manifolds always come with integral coefficients. Hence with regards to the question of realizing GKM graphs through manifolds, as is studied in this paper, this restriction is essential.
 
 Secondly, in \cite{MR2873096} GKM graphs carry signed structures, which is not necessarily the case for us (on the geometric side this corresponds to the fact that our manifolds are not necessarily almost complex). The adaptation of the definition of GKM fiber bundle is straight forward in this regard.
 
 Finally, in \cite{MR2873096} the axial function of the fiber graphs $\Gamma_p$ is
 considered as taking values in the rational span of the fiber weights, while we take
 labels in all of $\ZZ^m/\pm 1$ when talking about the isomorphisms between the fiber
 graphs (this reflects the fact, that fibers over $T$-fixed points are naturally
 $T$-manifolds). After intersecting said span with $\ZZ^m$, an isomorphism of this
 sublattice can always be extended to an automorphism of $\ZZ^m$. 	Conversely, any
 isomorphism $\ZZ^m\to \ZZ^m$ as in \Cref{defn:GKMfiberbundle} necessarily maps the span
 of the fiber weights to itself. Therefore, \Cref{defn:GKMfiberbundle} is indeed
 equivalent to \cite{MR2873096} in this regard.
\end{rem}

As all fiber graphs $\Gamma_p$ of a GKM fiber bundle $\Gamma'\to B$ are isomorphic, we call this GKM graph $\Gamma$, defined up to isomorphism, the fiber of $\Gamma'\to B$, and speak about a GKM fiber bundle $\Gamma\to \Gamma'\to B$.

Part $(i)$ of  the following definition was given in \cite[Definition 3.2]{MR4634089}, as an
intermediate stage between GKM fibrations between unsigned and signed GKM graphs. It will
be of relevance for us below in Section \ref{sec:Realizability} in the context of GKM
fiber bundles whose fiber is a  generalized flag manifold. Part $(ii
)$ is an adaptation of $(i)$ to the concept of GKM fiber bundles.

\begin{defn}\label{defn:fiberwise signed}
  \begin{enumerate}[label=$(\roman*)$]
\item
Let $\pi\colon(\Gamma,\alpha)\rightarrow (B,\alpha_B)$ be a GKM fibration, respectively a GKM fiber bundle. Let $F\subset E(\Gamma)$ be the set of vertical edges and $\tilde{\alpha}\colon F\rightarrow \mathbb{Z}^m$ a lift of $\alpha\colon E(\Gamma)\rightarrow \mathbb{Z}^m/\pm 1$ satisfying $\tilde{\alpha}(e)=-\tilde{\alpha}(\overline{e})$ for all $e\in F$. Then we call $\pi$, together with $\tilde{\alpha}$, a \emph{fiberwise signed GKM fibration respectively fiber bundle} if the connections $\nabla$ and $\nabla^B$ as in Definition \ref{defn:GKMfibration} respectively \ref{defn:GKMfiberbundle} can be chosen in a way such that $\tilde{\alpha}(\nabla_{e}e')\equiv \tilde{\alpha}(e')\mod \alpha(e)$ for any $e'\in F$ and $e\in E(\Gamma)$.
\item If $\pi$ is additionally a GKM fiber bundle we call it fiberwise signed if the
	automorphisms $\Phi_e$ for every horizontal edge $e$ (see Definition
	\ref{defn:GKMfiberbundle}) can be chosen to respect the signed structure, i.e.\ the
	associated linear transformation $\Psi_e\colon \Z^m\rightarrow \Z^m$
	satisfies $\tilde{\alpha}\circ \Phi_e = \Psi_e\circ \tilde{\alpha}$.
\end{enumerate}
\end{defn}

\subsection{GKM actions on homogeneous spaces}\label{sec:homspaces}

In this section we summarize the main result of \cite{MR2218848}, which describes the GKM graph of equal rank homogeneous spaces in terms of the root systems of the occurring Lie groups.

Consider a homogeneous space $G/K$, where $K\subset G$ is a compact connected subgroup of
a compact connected semisimple Lie group. We assume that $G$ and $K$ have the same rank,
so that a maximal torus $T\subset K$ is at the same time a maximal torus of $G$. On $G/K$
we consider the $T$-action by left multiplication. We denote by $\Delta_K\subset
\Delta_G\subset \ZZ_\mft^*$ the root systems of $K$ and $G$ with respect to $T$,
respectively, and put $\Delta_{G,K}:=\Delta_G\setminus\Delta_K$.  For $\alpha\in \Delta_G$
we denote the corresponding reflection by $\sigma_\alpha$, thus
\[
	\sigma_{\alpha}(\beta) = \beta - \frac{2 \langle \alpha, \beta  \rangle}{|\alpha|^{2}}
	\alpha
\]
with respect to the Killing form $\langle \cdot, \cdot \rangle$. We denote by $W(G)$ and $W(K)$ the Weyl
groups of $G$ and $K$; these may be defined as the quotients $N_G(T)/T$ and $N_K(T)/T$ of
the normalizers of $T$. We also consider them as finite groups acting on $\mft$, or
$\mft^*$, via the isomorphism $\mft\cong \mft^*$ given by the Killing form. Explicitly,
$w=gT\in W(G)$ acts on $\mft^*$ by $w\cdot \alpha:=\alpha\circ \Ad_g^{-1}$. In this way,
the reflections $\sigma_\alpha$ become elements of $W(G)$. Then from 
\cite[Theorem 2.4]{MR2218848} and \cite[Section 2.2.7]{MR2218848} we have
\begin{theorem}\label{thm:GHZ}
The $T$-action on $G/K$ is of GKM type. Its GKM graph $\Gamma$ is as follows:
\begin{enumerate}[label=(\alph*)]
    \item $V(\Gamma) = W(G)/W(K)$. We denote elements in $V(\Gamma)$ by $[w]:=wW(K)$, with $w\in W(G)$.
    \item For any $[w]\in V(\Gamma)$, we have $E(\Gamma)_{[w]} = \Delta_{G,K}/\pm 1$. Explicitly, for any $\alpha\in \Delta_{G,K}/\pm 1$, there is an edge connecting $[w]$ and $[w\sigma_\alpha]$ with label $w\cdot \alpha$.
		\item There is a \emph{canonical connection} defined as follows:
	Let $[w] \in V(\Gamma)$ and let $e \in E(\Gamma)_{[w]}$, such that $e$ joins $[w]$ and
	$[w \sigma_{\alpha}]$ with label $w \cdot \alpha$. If $f \in E(\Gamma)_{[w]}$ is the
	edge joining $[w]$ and $[w \sigma_{\beta}]$ with label $w \cdot \beta$ then we set 
	$\nabla_{e} f $ as the edge joining $[w \sigma_{\alpha}]$ and $[w \sigma_{\alpha}
	\sigma_{\beta}]$ with label $w\sigma_\alpha\cdot \beta$. 
\end{enumerate}
\end{theorem}

\begin{remark}\label{rem: canonical connection is compatible}
We point out that the canonical connection is indeed compatible with the axial 
function $\alpha$ defined by Theorem \ref{thm:GHZ} (b). From the definitions we have
\begin{align*}
\alpha(\nabla_{e}f) &=	(w \sigma_{\alpha}) \cdot \beta \\
&	= \Ad_{w}^{\ast}\left(\beta - \frac{2\langle \alpha,\beta\rangle}{|\alpha|^2} \cdot \alpha\right) 
= w \cdot \beta - \frac{2\langle \alpha,\beta\rangle}{|\alpha|^2} (w \cdot \alpha) =
	\alpha(f)- \frac{2\langle \alpha,\beta\rangle}{|\alpha|^2}  \cdot\alpha(e).
\end{align*}
\end{remark}

Below, we will be only interested in the case that $K=T$, i.e, the $T$-action on the full flag manifold $G/T$. Observe that in this case, as $V(\Gamma)=W(G)$, the GKM graph has no multiple edges.

Below, we will be interested in the case that $K=T$ is a maximal torus of
$G$. In this case, from \cite[Section IV.5]{MR1491979} $G/T$ admits a $G$-invariant Kähler structure, determined uniquely by a choice of positive roots $\Delta_+\subset \Delta_G$. Thus, the GKM graph $\Gamma$ of the $T$-action on $G/T$ obtains a signed structure, which was described explicitly in \cite[Section 3]{MR2218848}: the axial function $\alpha$, on the oriented edge $e$ from $[w]$ to $[w\sigma_\alpha]$, with $\alpha\in \Delta_+$, is given by $\alpha(e)=w\cdot \alpha$.

\subsection{Quasitoric manifolds}\label{sec:quasitoric}
In this section we recall the notion of a (strongly) quasitoric manifold and prove a lemma that will be needed in the proof of our realization result when considering smooth structures on the construction. We call two $T$-spaces \emph{twisted} equivariantly homeomorphic/diffeomorphic if they are equivariantly homeomorphic/diffeomorphic after pulling back one of the actions along an automorphism of $T$.

\begin{defn}A $T^k$-action on a $2k$-dimensional manifold $M$ is called \emph{locally standard} if every point $M$ admits an open neighborhood which is twisted equivariantly diffeomorphic to an open subset of $\CC^k$, equipped with the standard $T^k$-action by componentwise multiplication.
\end{defn}

The orbit space of a locally standard action naturally carries the structure of a manifold with corners. We thus obtain two rather similar definitions: 

\begin{defn}
A locally standard $T^k$-action on a $2k$-dimensional compact manifold is called \emph{quasitoric} if its orbit space is homeomorphic to a simple polytope $P$, in such a way that $l$-dimensional orbits are mapped to the interior of a $l$-dimensional face of $P$. It is called \emph{strongly quasitoric} if this homeomorphism may be chosen to be a diffeomorphism of manifolds with corners, where we equip $P$ with its standard differentiable structure.
\end{defn}

In case the $T$-action is noneffective, with kernel $T'$, we will call it (strongly)
quasitoric if the action of $T/T'$ is (strongly) quasitoric.

Quasitoric manifolds were defined in \cite{MR1104531}, while the notion of strongly
quasitoric manifold was introduced in  \cite{MR3030690}; in the same paper it was shown
that in every equivariant homeomorphism class of quasitoric manifold of a quasitoric
manifold there is, up to equivariant diffeomorphism, a unique strongly quasitoric
manifold.

Given a quasitoric manifold $M$ with orbit space projection $\pi:M\to M/T=P$, the action induces a so-called \emph{characteristic function} $\lambda$: for a facet $F$ of $P$, we denote by $\lambda(F)$ the common isotropy group of the points in $\pi^{-1}(F)$, which is a subcircle of $T$. Abstractly, one defines a characteristic function as a map $\lambda$ from facets of $P$ to the set of subcircles of $T$, satisfying that whenever $F_1,\ldots,F_k$ are facets of $P$ with $F_1\cap \ldots \cap F_k\neq \emptyset$, then the map $\lambda(F_1)\times \ldots \times \lambda(F_k)\to T$ is injective.

To any characteristic function $\lambda$ on $P$ one associates the \emph{canonical model} $P\times T/_\sim$, where
\[
(x,t)\sim (y,s)\Longleftrightarrow x=y, \textrm{ and } t^{-1}s \in \lambda(F_1)\times \ldots \times \lambda(F_k) \textrm{ whenever } x \in F_1\cap \ldots \cap F_k.
\] 
Any quasitoric manifold $M$ is equivariantly homeomorphic to its canonical model
\cite[Lemma 1.4]{MR1104531}.
Conversely, one may use the canonical model to find, for any given simple polytope $P$
with characteristic function $\lambda$, a quasitoric manifold, as one can construct a
smooth structure on the canonical model, see \cite[Definition 7.3.14]{MR3363157}. 

Note that the canonical model $P\times T/_\sim$ admits a canonical section $s:P\to P\times T/_\sim$ of the orbit space projection $\pi$, via $s(p)=[(p,e)]$. In this sense, describing an equivariant homeomorphism between any given quasitoric manifold $M$ with orbit space $P$ and its canonical model corresponds to choosing a section $s:P\to M$.

The following lemma is a consequence of \cite{MR4450678}.
\begin{lem}\label{lem: quasi smooth}
		Let $T=T^k$ be any compact torus and $M$ be a (not necessarily effective) strongly
		quasitoric $T$-manifold. Then there is a choice of the section $s \colon P \to M$ as
		above such that the resulting description $M=P\times T/_\sim$ has the following
		property: Let $U\subset P$ be open and $H\subset T$ be a subgroup containing all
		isotropies occurring over $U$. Then the map $\pi' \colon \pi^{-1}(U)\rightarrow
		T/H$ induced by the projection $P \times T \to T$ is smooth.
		\end{lem} 		
\begin{proof} The statement for non-effective actions follows directly from the effective case and we assume the action to be effective, so that the dimension of $M$ is $2k$.
As the action is locally standard, each point in $M$ has a neighborhood which is equivariantly diffeomorphic to $\CC^n\times T^{k-n}\times \RR^{k-n}$ where the action is described by some splitting $T\cong T^n\times T^{k-n}$ such that $T^{n}$ acts on $\CC^n$ linearly and effectively, $T^{k-n}$ acts in standard fashion on itself, and the action on the $\RR^{k-n}$ component is trivial. In these coordinates $\pi$ can be identified with the map
\[ \CC^n\times T^{k-n}\times \RR^{k-n}\longrightarrow \RR^n_{\geq 0}\times \RR^{k-n},\quad ((z_1,\ldots,z_n),t,y)\longmapsto ((|z_1|^2,\ldots,|z_n|^2),y).\]
We note that it is sufficient to prove the lemma in the case $\pi^{-1}(U)\rightarrow U$ is
of this form since smoothness is a local issue and dividing out bigger subgroups from the
target preserves smoothness. In this case, the largest occurring isotropy
is, with respect to the splitting of $T$, the subgroup $H=T^n$, so that $T/H\cong
T^{k-n}$.

The above description of the orbit map admits the standard section
\[s_0\colon ((x_1,\ldots,x_n),y)\longmapsto ((\sqrt{x_1},\ldots,\sqrt{x_n}),1,y).\]
By \cite[Theorem 3.3]{MR4450678} there is a global section
 $s\colon P\rightarrow M$ with the following property: for any coordinates as above there
 is a smooth function $f\colon U\rightarrow T$ such that $s(x,y)=f(x,y)s_0(x,y)$. The
 section $s$ induces an equivariant homeomorphism \[\CC^n\times
 T^{k-n}\times\RR^{k-n}\longleftarrow \RR^n_{\geq 0}\times \RR^{k-n}\times T/_\sim.\]
 Write $f(x)=(g(x),h(x))\in T^n\times T^{k-n}$. The map $\pi'$ in the lemma corresponds to
 the composition
\[\CC^n\times T^{k-n}\times\RR^{k-n}\longrightarrow \RR^n_{\geq 0}\times \RR^{k-n}\times T/_\sim\longrightarrow T^{k-n}\]
which can be checked to be the map $(z,t,y)\mapsto t\cdot h(\pi(z,t,y))^{-1}$. In
particular, it is smooth.
\end{proof}
	
In this paper, we will only be interested in four-dimensional quasitoric manifolds, which
are automatically strongly quasitoric, see \cite{MR3030690}.	

\section{From equivariant fiber bundles to graphs}\label{sec:equivariant fiber bundles}

In this section we prove the theorem below and discuss the necessity of the conditions in the theorem through counterexamples.

\begin{theorem}\label{thm: induced graph bundles}
Let $F\rightarrow M\rightarrow B$ be a fiber bundle in which all spaces are $T$-GKM manifolds and the projection map $M\to B$ is equivariant.
\begin{enumerate}[label = $(\roman*)$]
\item Then the induced maps on one-skeleta induce a GKM fibration.
\item If additionally, for each vertex $v$ of the graph of $M$, the weights of any horizontal edge adjacent to $v$ and any two vertical edges adjacent to $v$ are linearly independent, then transport along a horizontal edge induces a graph isomorphism between the respective fiber graphs.
\item Assume that for each vertex $v$ of the graph of $M$, the weights of any horizontal edge adjacent to $v$ and any three vertical edges adjacent to $v$ are linearly independent. Furthermore assume that any adjacent set of vertical edges can be completed to a basis of the weight lattice. Then the induced GKM fibration is a GKM fiber bundle.
\end{enumerate}
\end{theorem}

Statement $(i)$ was proved prior to the present paper in \cite{MR4634089}. However we did
not deal with the notion of GKM fiber bundles as it was not relevant in the low
dimensional case considered in \cite{MR4634089}. While there is a large overlap in the
arguments, the situation of GKM fiber bundles needs a refined viewpoint and additional
arguments in several places. Thus, it seems appropriate to rewrite the whole completed
argument here, as referencing refinements of certain steps in the old argument would be
rather inconvenient.

\begin{remark}\label{rem: pessimistic rem}
We point out that the conditions in $(ii)$ and $(iii)$ are not necessary for the induced
GKM fibration to be a GKM fiber bundle. In fact they are not satisfied by many of the
geometric realizations of GKM fiber bundles that we construct in this paper. However as it
turns out, it is not clear what the right geometric counterpart to the combinatorial
situation of a GKM fiber bundle should be. As illustrated by the
counterexamples \ref{ex: cool counterexample} and \ref{ex: long counterexample} below,
being a fiber bundle with equivariant projection is not sufficient. The counterexamples
are furthermore designed to show that several other natural geometric conditions fail to
assure the properties of a GKM fiber bundle on the combinatorial side.
\end{remark}

\begin{ex}\label{ex: cool counterexample}
We give an example of an equivariant fiber bundle of GKM manifolds such that the induced
graph fibration is the collapse of the dashed edges and the
identification of the solid edges in the picture

\begin{center}
\begin{tikzpicture}
\draw[very thick, dashed] (0,0) -- ++(1,1) -- ++(0,1) -- ++(-1.5,1.5) -- ++(-1,-1) -- ++ (0,-1) --++ (1.5,-1.5);

\draw[very thick, dashed] (7,0) -- ++(1,1) -- ++(0,1) -- ++(-1.5,1.5) -- ++(-1,-1) -- ++ (0,-1) --++ (1.5,-1.5);

\draw[very thick] (1,1) --++ (7,0);
\draw[very thick] (1,2) --++ (7,0);
\draw[very thick] (-1.5,1.5) --++ (7,0);
\draw[very thick] (-1.5,2.5) --++ (7,0);
\draw[very thick] (0,0) -- (6.5,3.5);
\draw[very thick] (-0.5,3.5) -- (7,0);

\node at (0,0)[circle,fill,inner sep=2pt]{};
\node at (1,1)[circle,fill,inner sep=2pt]{};
\node at (1,2)[circle,fill,inner sep=2pt]{};
\node at (-1.5,1.5)[circle,fill,inner sep=2pt]{};
\node at (-1.5,2.5)[circle,fill,inner sep=2pt]{};
\node at (-0.5,3.5)[circle,fill,inner sep=2pt]{};

\node at (7,0)[circle,fill,inner sep=2pt]{};
\node at (8,1)[circle,fill,inner sep=2pt]{};
\node at (8,2)[circle,fill,inner sep=2pt]{};
\node at (5.5,1.5)[circle,fill,inner sep=2pt]{};
\node at (5.5,2.5)[circle,fill,inner sep=2pt]{};
\node at (6.5,3.5)[circle,fill,inner sep=2pt]{};

\end{tikzpicture}
\end{center}
In particular transport along horizontal edges is not a graph automorphism between the
fibers; consequently, statements $(ii)$ and $(i i i)$ of Theorem \ref{thm: induced graph
bundles} are false without additional assumptions on the weights.

Let $T=T^2$ and $x,y\in \mathfrak{t}^*$ denote the dual basis to the standard basis. For some $z$ in the weight lattice we write $S^2_z$ for the sphere with the action given by $z$. For the construction we take $S^2_x\times S^2_y$ with the standard action. Denoting by $N,S$ the fixed points of $S^2$, we now blow up in the points $(N,N)$ and $(S,S)$, which gives a (quasi)toric  manifold $F$ with GKM graph
\begin{center}
\begin{tikzpicture}
\draw[very thick] (0,0) -- ++(1,-1) -- ++(2,0) -- ++(0,2) -- ++(-1,1) -- ++ (-2,0) --++ (0,-2);
\node at (0,0)[circle,fill,inner sep=2pt]{};
\node at (1,-1)[circle,fill,inner sep=2pt]{};
\node at (3,-1)[circle,fill,inner sep=2pt]{};
\node at (3,1)[circle,fill,inner sep=2pt]{};
\node at (2,2)[circle,fill,inner sep=2pt]{};
\node at (0,2)[circle,fill,inner sep=2pt]{};

\node at (-0.4,1){$y$};
\node at (-0.2,-0.6){$x-y$};
\node at (2,-1.4){$x$};
\node at (1,2.4){$x$};
\node at (3.4,0){$y$};
\node at (3.2,1.6){$x-y$};

\node at (3.4,-1.4){$(N,S)$};
\node at (-0.4,2.4){$(S,N)$};

\end{tikzpicture}
\end{center}

We now consider the base manifold $B=S^2_{x+y}$ and use an equivariant clutching construction to produce a fiber bundle $F\rightarrow M\rightarrow B$. Let $U=\{(s,\overline{s})\}\subset T$ denote the kernel of the $T$-action on $B$. Then any $U$-equivariant automorphism $\varphi\colon F\rightarrow F$ extends uniquely to a $T$-equivariant automorphism \[S^1_{x+y}\times F\longrightarrow S^1_{x+y}\times F\]
such that $(1,p)\mapsto (1,\varphi(p))$. Using this to glue two copies of $D_{x+y}^2\times F$ over the boundary circle gives an equivariant fiber bundle $F\rightarrow M\rightarrow B$. Hence to finish the construction we construct a $U$-equivariant automorphism of $F$ that swaps $(N,S)$ and $(S,N)$ while leaving the other fixed points fixed.

We consider the conjugation $p\mapsto \overline{p}$ on $S^2$ given by reflection at a
fixed hyperplane through $N,S$. In particular with respect to the standard action one has
$\overline{t\cdot p}=\overline{t}\cdot\overline{p}$. Set $\psi\colon S^2\times
S^2\rightarrow S^2\times S^2$ as $(p,q)\mapsto (\overline{q},\overline{p})$. Now $\psi$ is
$U$-equivariant and swaps $(N,S)$ and $(S,N)$. Finally, we will $U$-equivariantly isotope $\psi$ to
be the identity near $(N,N)$ and $(S,S)$. Having done that it induces the desired
automorphism on the blowup $F$.

We find a neighbourhood of $(S,S)$ which is $T$-equivariantly homeomorphic to
$D^4_{x,y}\subset \C_x\times \C_y$ such that $\psi$ corresponds to $(v,w)\mapsto
(\overline{w},\overline{v})$. Identifying $S^1\cong U$ via $s\mapsto
(s,\overline{s})$ the $U$-action is of the form $s\cdot (v,w)=(sv,\overline{s}w)$ where on
the right hand side we use standard complex multiplication. Let $\eta\colon I\rightarrow
\U(2)$ be a path from
\[\begin{pmatrix}
1 & 0\\ 0&1
\end{pmatrix}\quad\text{to}\quad \begin{pmatrix}
0&1\\1&0
\end{pmatrix}\]
which is constant near the endpoints. Denote by $C$ the map $(v,w)\mapsto (v,\overline{w})$. Then
\[D^4\longrightarrow D^4,\quad p\longmapsto C(\eta(\|p\|)C(p)).\]
Is $U$-equivariant. It is the identity near $0$ and extends by $\psi$ near the boundary. Modifying $\psi$ in this way and doing the analogous construction near $(N,N)$ yields the desired automorphism $\varphi$ of $F$.
\end{ex}

\begin{ex}\label{ex: long counterexample}
We show that, even in the situation of $(ii)$, of Theorem \ref{thm: induced graph
bundles}, one does not necessarily have a GKM fiber bundle on the combinatorial side as it
is possible to still violate the condition that the weights uniformly transform by an
automorphism when transported horizontally. In particular, part $(iii)$ of the theorem is
false without additional assumptions beyond those in $(ii)$. We give a counterexample in
two steps: the first construction will produce a $T^3$-equivariant fiber bundle
\[S^6\longrightarrow M\longrightarrow S^2\] such that the fiber over one fixed point is effective
while the other one is not. In particular, they will not be twisted equivariantly
homeomorphic and will not have isomorphic GKM graphs even though horizontal transport
(which is uniquely defined by Remark \ref{rem:horiztransportunique})
respects the graph structures. Furthermore it is interesting to observe from the
construction that the structure group of this example commutes with the action. Hence this
condition does not ensure that the result is a GKM fiber bundle. We point out that this
example could be simplified while retaining the same properties. However we chose to give
this non-minimal version as we need it for the second construction. Using
the total space $M$ we will construct a $T^3$-equivariant fiber bundle
\[M\longrightarrow N\longrightarrow S^2\]
which does not induce a GKM fiber bundle even though it satisfies the following condition
($\ast$): Each point in the base admits an invariant neighborhood $V$ and equivariant
local trivializations $N|_V\cong V\times M$ where the action on the right is the diagonal
action with respect to the actions on $V$ and $M$. Note that in this case fibers over
fixed points are $T$-equivariantly homeomorphic and in particular their GKM graphs are
abstractly isomorphic. However the underlying isomorphism of graphs will be a different
one than the one induced by horizontal transport. 

For the construction, let $T= T^3$ and $x,y,z\in \mathfrak{t}^*$ be weights dual to the standard basis. We consider the $T$-space $S^6_{x,y,x+2y}\subset \C^3\oplus \R$ where the weights in the index describe the action on the $\C$-factors. We now glue $D^2_{x-y-z}\times S^6_{x,y,x+2y}$ to $D^2_{x-y-z}\times S^6_{y+z,x-z,x+2y}$ over the boundary along the $T^3$-equivariant automorphism
\[S^1_{x-y-z}\times S^6_{x,y,x+2y} \longrightarrow S^1_{x-y-z}\times S^6_{y+z,x-z,x+2y},\quad (s,(v_1,v_2,v_3,h))\longmapsto (s,(\overline{s}v_1, sv_2, v_3,h))\]
Using this we build a $T$-equivariant fiber bundle $S^6\rightarrow M\rightarrow S^2$, whose GKM graph fibration is
\begin{center}
\begin{tikzpicture}
\node at (0,0)[circle,fill,inner sep=2pt]{};
\node at (4,0)[circle,fill,inner sep=2pt]{};

\node at (2,0.9) {$y+z$};
\node at (2,0.3) {$x-z$};
\node at (2,-0.4) {$x+2y$};

\node[inner sep = 0pt] (a) at (0,0){};
\node[inner sep = 0pt] (b) at (4,0){};

\draw[very thick] (a) -- (b);
\draw[very thick] (a) to[out=35, in=145] (b);
\draw[very thick] (a) to[out=-35, in=-145] (b);

\node at (0,-4)[circle,fill,inner sep=2pt]{};
\node at (4,-4)[circle,fill,inner sep=2pt]{};

\node at (2,-3.1) {$x$};
\node at (2,-3.7) {$y$};
\node at (2,-4.4) {$x+2y$};

\node[inner sep = 0pt] (c) at (0,-4){};
\node[inner sep = 0pt] (d) at (4,-4){};

\draw[very thick] (c) -- (d);
\draw[very thick] (c) to[out=35, in=145] (d);
\draw[very thick] (c) to[out=-35, in=-145] (d);

\draw[very thick] (a) -- (c);
\draw[very thick] (b) -- (d);
\node at (-1,-2) {$x-y-z$};
\node at (5,-2) {$x-y-z$};

\draw[very thick, ->] (6.5,-2) --++ (2,0);

\draw[very thick] (9.5,0) --++ (0,-4);
\node at (10.5,-2) {$x-y-z$};
\node at (9.5,0)[circle,fill,inner sep=2pt]{};
\node at (9.5,-4)[circle,fill,inner sep=2pt]{};

\end{tikzpicture}
\end{center}

Note that the upper fiber is effective since the weights are a basis of the weight lattice while the action on the lower fiber has a one-dimensional kernel. This concludes the construction of $M$.

Now we construct an equivariant fiber bundle with fiber $M$ over $S^2_z$. We do this by gluing two copies $D^2_z\times M$ with the diagonal action along a $T^3$-equivariant automorphism $S^1_z\times
 M\rightarrow S^1_z\times M$. To find the latter we consider a $U$-equivariant automorphism $\varphi\colon M\rightarrow M$, where $U$ is the kernel of the weight $z$ (i.e.\ the subtorus $T^2\subset T^3$ given by the first two components). For the gluing we use the unique $T^3$-equivariant map mapping $(1,p)\mapsto (1, \varphi(p))$.
 
To define $\varphi$ we observe that with the restricted $U$-action (i.e., setting $z=0$),
the two pieces glued for the construction of $M$ are $D^2_{x-y}\times S^6_{x,y,x+2y}$ and
$D^2_{x-y}\times S^6_{y,x,x+2y}$. Between them we have the
$U$-equivariant map
\[ (p,(v_1,v_2,v_3,h))\longmapsto (p,(v_2,v_1,v_3,h))\]
which we can use to swap the two pieces. This is compatible with the gluing and defines a $U$-equivariant automorphism $\varphi$ of $M$ which swaps the two fibers. This gives the desired $T^3$-equivariant fiber bundle $M\rightarrow N\rightarrow S^2_z$ with GKM graph

\begin{center}
\begin{tikzpicture}
\node at (0,0)[circle,fill,inner sep=2pt]{};
\node at (4,0)[circle,fill,inner sep=2pt]{};

\node at (2,0.9) {$y+z$};
\node at (2,0.3) {$x-z$};
\node at (2,-0.4) {$x+2y$};

\node[inner sep = 0pt] (a) at (0,0){};
\node[inner sep = 0pt] (b) at (4,0){};

\draw[very thick] (a) -- (b);
\draw[very thick] (a) to[out=35, in=145] (b);
\draw[very thick] (a) to[out=-35, in=-145] (b);

\node at (0,-4)[circle,fill,inner sep=2pt]{};
\node at (4,-4)[circle,fill,inner sep=2pt]{};

\node at (2,-3.1) {$x$};
\node at (2,-3.7) {$y$};
\node at (2,-4.4) {$x+2y$};

\node[inner sep = 0pt] (c) at (0,-4){};
\node[inner sep = 0pt] (d) at (4,-4){};

\draw[very thick] (c) -- (d);
\draw[very thick] (c) to[out=35, in=145] (d);
\draw[very thick] (c) to[out=-35, in=-145] (d);

\draw[very thick] (a) -- (c);
\draw[very thick] (b) -- (d);
\node at (-1,-2) {$x-y-z$};


\node at (11,-2) {$x-y-z$};

\node at (10,0)[circle,fill,inner sep=2pt]{};
\node at (6,0)[circle,fill,inner sep=2pt]{};

\node at (8,0.9) {$y+z$};
\node at (8,0.3) {$x-z$};
\node at (8,-0.4) {$x+2y$};

\node[inner sep = 0pt] (a1) at (6,0){};
\node[inner sep = 0pt] (b1) at (10,0){};

\draw[very thick] (a1) -- (b1);
\draw[very thick] (a1) to[out=35, in=145] (b1);
\draw[very thick] (a1) to[out=-35, in=-145] (b1);

\node at (10,-4)[circle,fill,inner sep=2pt]{};
\node at (6,-4)[circle,fill,inner sep=2pt]{};

\node at (8,-3.1) {$x$};
\node at (8,-3.7) {$y$};
\node at (8,-4.4) {$x+2y$};

\node at (1,-3) {$z$};
\node at (1,-1) {$z$};
\node at (9,-3) {$z$};
\node at (9,-1) {$z$};

\node[inner sep = 0pt] (c1) at (6,-4){};
\node[inner sep = 0pt] (d1) at (10,-4){};

\draw[very thick] (c1) -- (d1);
\draw[very thick] (c1) to[out=35, in=145] (d1);
\draw[very thick] (c1) to[out=-35, in=-145] (d1);

\draw[very thick] (a1) -- (c1);
\draw[very thick] (b1) -- (d1);

\draw[very thick] (a) -- (c1);
\draw[very thick] (b) -- (d1);
\draw[very thick] (c) -- (a1);
\draw[very thick] (d) -- (b1);

\end{tikzpicture}
\end{center}
As the weights of any horizontal edge and any two adjacent vertical edges
are linearly independent, horizontal transport is uniquely defined by Remark
\ref{rem:horiztransportunique}. It corresponds to the automorphism $\varphi$ on the graph
level. There is no automorphism of $\mathfrak{t}^*$ to make this compatible with the
weights: the weights of any triple edge need to get mapped onto those of the target triple
edge but in one case their span is $2$-dimensional while in the other it is
$3$-dimensional. We point out that condition ($\ast$) is fulfilled since the inclusions of the
two individual summands used for the gluing in the construction of $N$ are equivariant
trivializations.

\end{ex}

We now turn towards the proof of Theorem \ref{thm: induced graph bundles} and prove some preliminary lemmas.

\begin{lem}\label{lem: fiberstuff}
Let $F\rightarrow M\rightarrow B$ be an equivariant fiber bundle of GKM manifolds and $U\subset T$ be a codimension $1$ subtorus such that $M^U$ has a connected component $S$ with $F\cap S\neq \emptyset$ and $S\not\subset F$. Then $F^U = F^T$.
\end{lem}

\begin{proof}
The component $S$ is an invariant $2$-sphere of $M$ and it intersects $F$ in a point $p$ of isotropy codimension $1$ or $0$. Since $F$ is GKM, $p$ is contained in an invariant $2$-sphere contained in $F$ which is hence distinct from $S$. Thus two invariant $2$-spheres of $M$ meet in $p$, implying that $p$ is a fixed point.
For every $q\in F^T$ one has a $T$-equivariant splitting $T_q M= T_q F\oplus T_{\overline{p}} B$ where $\overline{p}$ is the image of $p$ in $B$. By what we have established the right hand summand contains an irreducible representation whose kernel contains $U$ and due to linear independence of the weights in $M$ it follows that no such weight occurs in the left hand summand. It follows that $F$ contains no invariant $2$-sphere whose isotropy contains $U$ and thus $F^U= F^T$.
\end{proof}

\begin{lem}\label{lem: automorphismstuff}
Let $\Gamma,\Gamma'$ be GKM graphs with labels in $\Z^n/\pm 1$ and the property that at
each point the edge weights are part of a $\Z$-basis of $\Z^n$. Furthermore let
$p\colon\Z^n\rightarrow \Z^k$ be a projection such that $\Gamma$ and
$\Gamma'$, with the induced labels in $\Z^k/\pm 1$, are GKM$_3$ graphs. Then
the linear transformation $\Z^k\rightarrow \Z^k$ associated to an
isomorphism $\Gamma\rightarrow \Gamma'$ of the $\Z^k$-GKM graphs can be
lifted to $\Z^n$ to obtain an isomorphism of the $\Z^n$-GKM graphs.
\end{lem}

\begin{proof}
We denote by $\alpha$ the $\Z^n/\pm 1$-axial functions on both graphs and by $\overline{\alpha}$ the induced $\Z^k/\pm 1$-axial functions. For simplicity, when it is clear from the context, we will occasionally use the same notation, for choices of lifts with fixed signs.
On both $\Gamma$ and $\Gamma'$, there is a unique connection $\nabla$ compatible with
$\alpha$. Note that it is also compatible with $\overline{\alpha}$. Let $\Phi\colon
\Gamma\rightarrow \Gamma'$ be an isomorphism of graphs and $\overline{\Psi}\colon
\Z^k\rightarrow \Z^k$ an automorphism such that
$\overline{\alpha}(\Phi(e))=\overline{\Psi}(\overline{\alpha}(e))$ for all edges. For any
adjacent edges $e,f\in E(\Gamma)$ and any sign choice for the labels
there is some unique $l\in \Z$ and choice of sign for
$\alpha(\nabla_e(f))$ such that $\alpha(\nabla_e(f))=\alpha(f)+l\alpha(e)$. Hence,
\[\overline{\alpha}(\Phi(\nabla_e(f)))=
\overline{\Psi}(\overline{\alpha(f)})+l\overline{\Psi}(\overline{\alpha}(e))\]
However we also have
\[\overline{\alpha}(\nabla_{\Phi(e)}\Phi(f))\equiv \overline{\Psi}(\overline{\alpha(f)})\mod\overline{\Psi}(\overline{\alpha}(e))\]
and it follows from the GKM$_3$ condition that $\nabla_{\Phi(e)}\Phi(f)=\Phi(\nabla_e f)$.

Choose a vertex $p$ and signs for all edge labels emanating from $p$.
We make the corresponding sign choices for the edge labels at $\Phi(p)$. By assumption
these can be completed to bases of $\ZZ^n$. Let $\Psi\colon \Z^n\rightarrow\Z^n$ be an
automorphism that maps the basis at $p$ onto an appropriate choice of basis at $\Phi(p)$
according to the map $\Phi$ on those basis elements that are weights of $E_p$. By
construction we have
\begin{equation}
\alpha(\Phi(e))=\Psi(\alpha(e)) \label{eq:projcomp}
\end{equation} for all $e\in E_p$. We claim that in fact \eqref{eq:projcomp} already holds for all edges in $\Gamma$. This follows from the following claim: Suppose \eqref{eq:projcomp} holds for adjacent edges $e,f$, then it holds for $h:=\nabla_e(f)$. 

To prove the claim, first fix sign choices for $\alpha(e)$ and
$\alpha(f)$ and the induced sign choices for $\alpha(\Phi(e)),\alpha(\Phi(f))$. Since
$\Psi$ commutes with $\nabla$, there are unique $l,m\in \Z$ and sign
choices for $\alpha(h),\alpha(\Phi(h))$ such that
\[\alpha(h)=\alpha(f)+l\alpha(e),\qquad \alpha(\Phi(h))=\alpha(\Phi(f))+m\alpha(\Phi(e)).\]
We obtain induced sign choices for $\overline{\alpha}$ and have
\begin{align*}
\overline{\Psi}(\overline{\alpha}(f)+l\overline{\alpha}(e))=\overline{\Psi}(\overline{\alpha}(h))=\pm \overline{\alpha}(\Phi(h))=\pm \overline{\alpha}(\Phi(f))\pm m\overline{\alpha}(\Phi(e)) = \pm
\overline{\Psi}(\overline{\alpha}(f)\pm m\overline{\alpha}(e))
\end{align*}
which implies $m=l$. Now since \eqref{eq:projcomp} holds for $e,f$ we obtain $\alpha(\Phi(h))=\Psi (\alpha(h))$ as desired.
\end{proof}

\begin{proof}[Proof of Theorem \ref{thm: induced graph bundles}]
Let $p\in M^T$ be a fixed point, denote $\overline{p}:=\pi(p)$ and let $F_{\overline{p}}$
the fiber over $\overline{p}$. Then $T_p M$ splits into irreducible summands belonging to
the invariant $2$-spheres in $M$ emanating from $p$. The spheres belonging to summands in
$T_p F_{\overline{p}}$ correspond to invariant spheres in $F_{\overline{p}}$ and hence to
vertical edges. We need to show that
\begin{enumerate}[label = (\alph*)]
\item Any invariant $2$-sphere $S$ belonging to an irreducible summand outside of
	$T_pF_{\overline{p}}$ connects $p$ to a $T$-fixed point $q$ in some fiber $F_{\overline{q}}$
	different from $F_{\overline{p}}$. In particular this implies that the induced map is a
	graph fibration.
\item It is possible to define a connection on the GKM graph of $M$ which, along any edge, respects horizontal and vertical edges, and which is such that the connection on horizontal edges lifts a given connection from the base graph. 
\item Under the additional condition from $(ii)$ the connection from (b) is unique in horizontal direction, and induces an isomorphism of graphs of $F_{\overline{p}}$
	and $F_{\overline{q}}$.
\item Under the additional condition of $(iii)$ the isomorphism of part (b) admits a compatible automorphism of the weight lattice such that it becomes an isomorphism of the GKM graphs of $F_{\overline{p}}$ and $F_{\overline{q}}$.

\end{enumerate}

Let $T_pM= T_pF_{\overline{p}}\oplus H_p$ an invariant decomposition and choose $S$ as
above. Then the projection $\pi\colon M\rightarrow B$ induces an equivariant isomorphism
$H_p\cong T_{\overline{p}} B$. Hence $T_p S$ corresponds to an irreducible summand in
$T_{\overline{p}} B$ and hence to an irreducible $2$-sphere $\overline{S}\subset B$
with $\overline{p}\in \overline{S}$. Let $\overline{q}$ denote the other
fixed point of $\overline{S}$.

Let $\gamma\colon I\rightarrow \overline{S}$ be an embedded path from $\overline{p}$ to $\overline{q}$ and let $U\subset T$ be the kernel of the $T$-action on $T_p S$ or equivalently the generic isotropy on $S$ and $\overline{S}$. Then we may regard $I$ as a $U$-equivariant map by equipping $I$ with the trivial $U$-action. Thus the equivariant pullback of the bundle along $\gamma$ defines an equivariant subbundle $F\rightarrow E\rightarrow I$ embedding into $M$. Since $\gamma$ is $U$-equivariantly nullhomotopic it follows that there is a $U$-equivariant trivialization
$E\cong F_{\overline{p}}\times I$. The subspace corresponding to $F_{\overline{p}}\times\{1\}$ is the fiber $F_{\overline{q}}$ over $\overline{q}$ where we set $q$ to be the point corresponding to $(p,1)$. It follows that $p$ and $q$ lie in a common connected component of $E^U\subset M^U$. Since connected components of $M^U$ are invariant $2$-spheres, we deduce that $S$ contains $q$. Now it follows from Lemma \ref{lem: fiberstuff} that $q$ is indeed a $T$-fixed point. We have proved (a).

For the proof of (b) we first construct the connection along horizontal edges. We point out that $T_pF_{\overline{p}}$ and $T_qF_{\overline{q}}$ are isomorphic as $U$-representations. To define the transport of vertical edges along a horizontal sphere $S$ we may choose any bijection between irreducible summands of these representations such that the $U$-isomorphism type is preserved (corresponding to a bijection between vertical edges such that labels fulfil the congruence relations in the definition of a compatible connection). To transport horizontal edges along $S$, we fix a connection on the GKM graph of $B$ and use the unique lift to horizontal edges. 

For the construction of the connection along vertical edges, let $p'$ be another fixed point in $F_{\overline{p}}$ such that there is a vertical edge between $p$ and $p'$.
One has $T_p M\cong T_p F_{\overline{p}}\oplus T_{\overline{p}} B$
and $T_{p'} M\cong T_{p'} F_{\overline{p}}\oplus T_{\overline{p}} B$ as
$T$-representations. Also, if $H\subset T$ denotes the kernel of the weight of the edge connecting $p,p'$ then $T_pF_{\overline{p}}$ and $T_{p'}F_{\overline{p}}$ agree as $H$-representations. Thus there is a bijection between the irreducible factors of $T_pM$ and $T_{p'}M$ that respects the decomposition in horizontal and vertical factors, preserves the $H$-isomorphism type on vertical factors and the $T$-isomorphism type on horizontal factors. This completes the proof of (b). We have shown $(i)$, i.e., that we have a GKM fibration.

For the proof of (c), i.e., of $(ii)$, note that by the arguments from (a) the fixed trivialization $E\cong F_{\overline{p}}\times I$ gives a $U$-equivariant homeomorphism $ F_{\overline{p}}\cong F_{\overline{p}}\times\{0\}\cong F_{\overline{p}}\times\{1\}\cong F_{\overline{q}}$ which sends $T$-fixed points to $T$-fixed points connected to the original point by a horizontal edge whose principal isotropy is $U$. Under the additional assumption in (c) any two weights of the $T$-action in $F_{\overline{p}}$ are linearly independent modulo the weight of $S$. Hence the restricted $U$-action on $F_{\overline{p}}$ will again be GKM. The same holds for $F_{\overline{q}}$. The GKM graph of the $U$-action on the fibers is the restriction of the GKM graph of the $T$-action. In particular the $U$-equivariant homeomorphism $F_{\overline{p}}\cong F_{\overline{q}}$ respects the graph structures. Hence in the definition of the transport of vertical edges along horizontal edges in (b), there is a unique choice and it will automatically induce an isomorphism of the underlying graphs.

In (d) the conditions of $(iii)$ imply that the GKM graphs of the restricted $U$-actions on $F_{\overline{p}}$ and $F_{\overline{q}}$ are GKM$_3$. The claim now follows from Lemma \ref{lem: automorphismstuff}.
\end{proof}
	
\section{Automorphisms of the GKM graph of $G/T$}\label{sec:autos}

We denote by $\Gamma$ the GKM graph of the $T$-action on $G/T$, as described in Theorem \ref{thm:GHZ}. We can assume that the compact connected semisimple Lie group $G$ is simply-connected, as the homogeneous space $G/T$ does not change upon passing to a cover: this is because any subgroup of the (finite) center of a simply-connected $G$ is automatically contained in any maximal torus of $G$. 

 In this section, we give basic examples of automorphisms of $\Gamma$ (see Definition \ref{defn:gkmgraphiso}) and classify those that are compatible with the standard connection.
	\begin{example}\label{rightmult}
		Clearly, $G$-equivariant diffeomorphisms $f \colon G/T \to G/T$ also define
		$T$-equivariant ones. The former are in one-to-one correspondence with $N_G(T)/T$,
		where this correspondence is given by sending $f$ to its value $w$ on
		$eT$, or, the other way around, sending $w$ to the map $gT\mapsto gwT$, that is,
		right-multiplication with $w$. Such a map $f$ induces a graph automorphism
		$\Phi\colon \Gamma\to \Gamma$ in the following way:
		\begin{enumerate}[label=$(\roman*)$]
			\item A vertex $[w']$ is sent to the vertex $[w'w]$, where we consider $w$ as an element in $W(G)$.
			\item The edge $e$ between $[w']$ and $[w'\sigma_{\alpha}]$ (with label $w'\cdot \alpha$) is sent to the edge
				between $[w'w]$ and $[w'\sigma_{\alpha} w]$ (with label $w'\cdot \alpha$). 
				(Because $w'\sigma_\alpha w = w'ww^{-1}\sigma_\alpha w = w'w \sigma_{w^{-1}\cdot
				\alpha}$, by Theorem \ref{thm:realization} there is such an edge with label
				$w'w\cdot (w^{-1}\cdot \alpha) = w'\cdot \alpha$.) 
			\item Since the diffeomorphism is $T$-equivariant, we may define $\Psi$ to be the identity map. 
		\end{enumerate}
	\end{example}
	\begin{example}\label{Type1}
		We can also consider left-multiplication $L_{w_{0}} \colon G \to G$ with an element
		$w_{0}\in N_G(T)$. The multiplication clearly preserves $N_{G}(T)$, whence it descends
		to an automorphism $W(G) \to W(G)$, as well to an automorphism $G/T \to G/T$ which we
		denote both again by $L_{w_{0}}$. We have the equation
		\[
			L_{w_{0}}(t \cdot gT) = w_{0}tgT =w_{0}tw_{0}^{-1}w_{0}gT =
			c_{w_{0}}(t)L_{w_{0}}(gT),
		\]
		so $L_{w_{0}} \colon G/T \to G/T$ is not $T$--equivariant,
		but rather twisted $T$--equivariant with respect to $c_{w_{0}}$. Therefore, it induces a graph
		automorphism  $L_{w_{0}}:\Gamma\to \Gamma$ by $L_{w_{0}}([w]) = [w_{0}w]$, and
		sending the edge between $[w]$ and $[w\sigma_\alpha]$ labelled $w\cdot \alpha$ to
	the edge between $[w_{0}w]$ and $[w_{0}w\sigma_\alpha]$ labelled $w_{0}w\cdot \alpha$. It becomes
an automorphism of the GKM graph $\Gamma$ with respect to $\Psi(\beta):= \beta\circ
\Ad_{w_{0}}^{-1} = w_{0}\cdot \beta$. Note that $L_{w_{0}} \colon
\Gamma \to \Gamma$ depends only on the class of $w_{0}$ in $W(G)= N_{G}(T)/T$ and
therefore we denote the corresponding automorphism by $L_{[w_{0}]}$.

Observe further that this automorphism respects the signed structure on $\Gamma$ induced by a choice of positive roots $\Delta_+\subset \Delta_G$, cf.\ Section \ref{sec:homspaces} above.
	\end{example}
	There is one more natural example coming from certain automorphisms of $G$.
We remind the reader of the relation between $W(G) = N_{G}(T)/T$ and
	the reflections $\sigma_{\alpha}$. For each root $\alpha$ there is an element
	$w_{\alpha} \in N_{G}(T)$ such that 
	\[
		\Ad_{w_{\alpha}}(h_{\alpha}) = - h_{\alpha}\quad \text{ and } \quad
		\Ad_{w_{\alpha}}(X) = X
	\]
	where $h_{\alpha} \in \mathfrak t$ corresponds to $\alpha$ under the identification of
	$\mathfrak t$ and $\mathfrak t^{\ast}$ by the Killing form and for all $X$ in $\ker
	\alpha$. Thus $\Ad_{w_{\alpha}}^{\ast} = \sigma_{\alpha}$ (see \cite[Theorem 11.35]{MR3331229}).
	\begin{example}\label{ex:Type2}
	For an
automorphism $\psi \colon G \to G$ which sends $T$ to $T$  we have
\[
 \psi(w)\psi(t)\psi(w)^{-1} = \psi(wtw^{-1}) \in T
\]
for $w \in N_{G}(T)$ and all $t \in T$. Hence, $\psi$ induces an automorphism $\psi \colon
W(G) \to W(G)$.
With the above identification of $W(G)$ with a reflection group on $\mathfrak t^{\ast}$ we infer
$\psi(\sigma_{\alpha})=\Ad_{\psi(w_{\alpha})}^{\ast}$ and thus
$\psi(\sigma_{\alpha}) = \sigma_{\alpha \circ d \psi^{-1}}$. Moreover, $\psi$ induces a
diffeomorphism $\psi \colon G/T \to G/T$, by $\psi(gT) =\psi(g)T$.
		 As $\psi(tgT) = \psi(t)\psi(gT)$, this diffeomorphism is twisted
		 equivariant with respect to $\psi|_{T} \colon T\to T$.
  
				We obtain a graph automorphism $\Gamma\to \Gamma$ by declaring
			 $\Phi([w]):=[\psi(w)]$ and sending the edge between $[w]$ and $[w\sigma_\alpha]$
		 labelled $w\cdot \alpha$ to the edge between $[\psi(w)]$ and $[\psi(w\sigma_\alpha)]
	 = [\psi(w)\sigma_{\alpha\circ (d\psi)^{-1}}]$ labelled $\psi(w)\cdot (\alpha\circ
 (d\psi)^{-1})$. Hence $\Phi$ becomes an automorphism of the GKM graph $\Gamma$ with
respect to $\Psi( \beta):=\beta\circ (d\psi)^{-1}$. 		 As $\psi\colon W(G)\to W(G)$ is a
homomorphism, at least $e\in W(G)$ is fixed and so its corresponding vertex.
	\end{example}
	\begin{definition}
		We say that the automorphisms in Example \ref{Type1} are of \textbf{Type 1} and that
		those in Example \ref{ex:Type2} are of \textbf{Type 2}. The sets of each type form a
		group by concatenation. The group of Type 1 automorphisms are isomorphic to $W(G)$ and
		we denote the group of Type 2 automorphisms by $T_{2}(\Gamma)$.
		Having fixed a choice of positive roots $\Delta_+\subset \Delta_G$,
		we denote by $T_2^+(\Gamma)\subset T_2(\Gamma)$ the subgroup of those Type 2
	automorphism respecting the corresponding signed structure, i.e.\
	the equation $\alpha(\Phi(e))=\Psi(\alpha(e))$ from Definition \ref{defn:gkmgraphiso}
holds with signs.
	\end{definition}
	\begin{rem}
		At first glance, it might seem odd that right-multiplication featured in \ref{rightmult} is not relevant to us here. However, since multiplication with $w\in N_G(T)$
		from the right can be written as composition of left-multiplication with $w$ (which is of Type 1) and conjugation with $w^{-1}$ (which is of Type 2),
		right-multiplication is already covered by the composition of automorphisms of Type 1 and 2.
	\end{rem}
	The main result of this section is the statement that all graph automorphisms come from the examples mentioned before.
	But first, we need three lemmata, the second of which is standard and proven for
	completeness. An abstract GKM graph is called effective if the edge
	weights at any vertex span all of $\mathfrak{t}^*$.
	\begin{lem}\label{AutoRigidity}
		Any graph automorphism $\Phi\colon \Gamma\rightarrow \Gamma$, $\Psi\colon
		\Z^n\rightarrow \Z^n$ of any effective abstract GKM graph
		$(\Gamma,\alpha)$ which preserves a compatible
	connection $\nabla$ is uniquely determined by its value on one vertex $v$ and all edges emerging
	from it.
	\end{lem}
	\begin{proof}
		It is clear that the linear map $\psi_*\colon \mft^* \to \mft^*$ is uniquely determined, as the labels on the edges at a vertex span $\mft^*$.
We fix a connection $\nabla$ which is preserved by $\psi$. Let $v'$ be a vertex connected by the edge $e$ to $v$. By definition, $\psi$ maps $v'$ to the vertex $\psi(v')$ connected by $\psi(e)$ to $\psi(v)$.
		Let $e'\neq e$ be another edge emerging from $v'$. We may write it as $\nabla_e
		\hat{e}$, where $\hat{e}$ is some edge emerging from $v$. Since $\psi$	preserves
		$\nabla$, the edge $\psi(e')$ is determined by $\psi(e)$ and $\psi(\hat{e})$. Since
		$\Gamma$ was assumed to be connected, we are done.
	\end{proof}

	\begin{lem}\label{LieAlgAuto}
		Let $\mathfrak{g}$ be the Lie algebra of a compact semisimple Lie group, $\mathfrak{t}$ a maximal abelian subalgebra and
		$\phi\colon \mathfrak{t} \to \mathfrak{t}$ a linear automorphism such that $\phi^*:\mathfrak{t}^*\to \mathfrak{t}^*$ permutes the roots of $\mathfrak{g}$.  Then $\phi$ extends to an automorphism of $\mathfrak{g}$.
	\end{lem}
	\begin{proof}
		Consider the complexification $\mathfrak{g}^\CC$ of $\mathfrak{g}$ as well as its root space decomposition
		\[
		\mathfrak{g}^\CC = \mathfrak{t}^\CC \oplus \bigoplus_{\alpha\in \Delta} \mathfrak{g}_\alpha
		\]
		with respect to the Cartan subalgebra $\mathfrak{t}^\CC$. Moreover, by 
		\cite[Theorem VI.6.6]{MR1920389} we may choose, for each $\alpha\in \Delta$, a root vector $X_\alpha\in \mathfrak{g}_\alpha$ such that for all $\alpha,\beta\in \Delta$
		\begin{enumerate}
			\item $[X_\alpha,X_{-\alpha}] = H_\alpha$ 
			\item $[X_\alpha,X_\beta] = N_{\alpha,\beta} X_{\alpha+\beta}$ if $\alpha+\beta\in \Delta$
			\item $[X_\alpha,X_\beta]=0$ if $\alpha+\beta\neq 0$ and $\alpha+\beta\notin \Delta$.
		\end{enumerate}
		for certain specific constants $N_{\alpha,\beta}$ satisfying $N_{-\alpha,-\beta} =
		-N_{\alpha,\beta}$. Here, $H_\alpha$ is the unique element in the Cartan subalgebra
		such that $\alpha(H)=B(H,H_\alpha)$ for all $H\in \mathfrak{t}^\CC$, where $B$ is the
		Killing form of $\mathfrak{g}^\CC$. By \cite[Theorem VI.6.11]{MR1920389}, in particular
		Equation (6.12) in the proof of this theorem, we may assume that the compact real form
		$\mathfrak{g}$ of $\mathfrak{g}^\CC$ is given by
		\[
		\mathfrak{g} = \mathfrak{t} \oplus \bigoplus_{\alpha\in \Delta_+} \RR(X_\alpha - X_{-\alpha}) \oplus \bigoplus_{\alpha\in \Delta_+} \RR i (X_\alpha + X_{-\alpha}),
		\]
		where $\Delta_+\subset \Delta$ is a choice of positive roots. 
		
		Denote by $\Pi\subset \Delta_+$ the corresponding set of simple roots. By 
		\cite[Theorem 2.108]{MR1920389}, the isomorphism $\phi$ extends uniquely to a Lie algebra automorphism $\phi:\mathfrak{g}^\CC\to \mathfrak{g}^\CC$ satisfying $\phi(X_{\alpha}) = X_{\alpha'}$ for all $\alpha\in \Pi$, where $\alpha' = (\phi^*)^{-1}(\alpha)$. As the Killing form is invariant under any Lie algebra automorphism, it follows that $\phi(H_\alpha) = H_{\alpha'}$ for all $\alpha\in \Delta$. For all $\alpha\in \Pi$, condition 1.\ above implies that \[
		H_{\alpha'} = \phi(H_\alpha) = \phi([X_\alpha,X_{-\alpha}]) = [X_{\alpha'},\phi(X_{-\alpha})],
		\]
		and hence, 
		$\phi(X_{-\alpha}) = X_{-\alpha'}$ for all simple $\alpha$.
		
		We claim that this automorphism $\phi$ automatically sends the compact real form $\mathfrak{g}$ to itself. To this end, take any $\alpha\in \Delta_+$ and write $\alpha = \alpha_1+\ldots + \alpha_k$ as a sum of simple roots (not necessarily distinct). We have
		\[
		[X_{\alpha_1},[X_{\alpha_2},\ldots,[X_{\alpha_{k-1}},X_{\alpha_k}]\ldots]] = N_{\alpha_1,\alpha_2 + \ldots + \alpha_k}\cdot\ldots\cdot N_{\alpha_{k-1},\alpha_k} X_{\alpha},
		\]
		where the coefficient of $X_\alpha$ does not vanish. Thus,
		\begin{align*}
			\phi(X_\alpha) &= \frac{1}{N_{\alpha_1,\alpha_2 + \ldots + \alpha_k}\cdot\ldots\cdot N_{\alpha_{k-1},\alpha_k}} [X_{\alpha'_1},[X_{\alpha'_2},\ldots,[X_{\alpha'_{k-1}},X_{\alpha'_k}]\ldots]]\\
			&=\frac{N_{\alpha'_1,\alpha'_2 + \ldots + \alpha'_k}\cdot\ldots\cdot N_{\alpha'_{k-1},\alpha'_k}}{N_{\alpha_1,\alpha_2 + \ldots + \alpha_k}\cdot\ldots\cdot N_{\alpha_{k-1},\alpha_k}} X_{\alpha'_1+\ldots + \alpha'_k}\\
			&= \frac{N_{\alpha'_1,\alpha'_2 + \ldots + \alpha'_k}\cdot\ldots\cdot N_{\alpha'_{k-1},\alpha'_k}}{N_{\alpha_1,\alpha_2 + \ldots + \alpha_k}\cdot\ldots\cdot N_{\alpha_{k-1},\alpha_k}} X_{\alpha'}.
		\end{align*}
		Analogously, we have
		\[
		[X_{-\alpha_1},[X_{-\alpha_2},\ldots,[X_{-\alpha_{k-1}},-X_{\alpha_k}]\ldots]] = N_{-\alpha_1,-\alpha_2 - \ldots - \alpha_k}\cdot\ldots\cdot N_{-\alpha_{k-1},-\alpha_k} X_{-\alpha},
		\]
		and because we showed above that $\phi(X_{-\alpha_i}) = X_{-\alpha'_i}$, we may compute analogously to above
		\begin{align*}
			\phi(X_{-\alpha}) &= \frac{N_{-\alpha'_1,-\alpha'_2 - \ldots -\alpha'_k}\cdot\ldots\cdot N_{-\alpha'_{k-1},-\alpha'_k}}{N_{-\alpha_1,-\alpha_2 - \ldots - \alpha_k}\cdot\ldots\cdot N_{-\alpha_{k-1},-\alpha_k}} X_{-\alpha'}\\
			&=\frac{N_{\alpha'_1,\alpha'_2 + \ldots + \alpha'_k}\cdot\ldots\cdot N_{\alpha'_{k-1},\alpha'_k}}{N_{\alpha_1,\alpha_2 + \ldots + \alpha_k}\cdot\ldots\cdot N_{\alpha_{k-1},\alpha_k}} X_{-\alpha'}.
		\end{align*}
		This implies that $\phi(\mathfrak{g})$ is contained in, and hence equal to $\mathfrak{g}$.
	\end{proof}

\begin{lem}\label{lem: type 1 and 2 preserve can connection}
Automorphisms of Type 1 and 2 preserve the canonical connection on the GKM graph $\Gamma$ of $G/T$, see Theorem
\ref{thm:GHZ}.\end{lem}

\begin{proof}
We denote by $e$ the edge joining $[w]$ and $[w \sigma_{\alpha}]$, with label $w\cdot \alpha$, by $f$ the edge
	joining $[w]$ and $[w \sigma_{\beta}]$, with label $w\cdot\beta$. Then by the definition  in Theorem \ref{thm:GHZ},
	$\nabla_{e} f$ is the edge joining $[w \sigma_{\alpha}]$ and $[w \sigma_{\alpha}
	\sigma_{\beta}]$, with label $w\sigma_\alpha\cdot \beta$.

 Let $\Phi$ be a Type 1 automorphism which is induced by left multiplication by $w_{0} \in
 W(G)$. Then $\Phi(e)$ connects $[w_{0}w]$ and $[w_{0}w\sigma_{\alpha}]$, $\Phi(f)$ connects
 $[w_{0}w]$ and $[w_{0}w \sigma_{\beta}]$, and therefore $\nabla_{\Phi(e)}\Phi(f)$ joins
$[w_{0}w\sigma_\alpha]$ and $[w_{0}w \sigma_{\alpha}\sigma_{\beta}]$. Clearly this is equal to
$\Phi(\nabla_{e}f)$.

For Type 2 we have that $\Phi(e)$ joins $[\psi(w)]$ and $[\psi(w) \sigma_{\alpha \circ
	d \psi^{-1}}]$ and $\Phi(f)$ joins $[\psi(w)]$ and $[\psi(w) \sigma_{\beta \circ d
\psi^{-1}}]$ (we use the notation of Example \ref{ex:Type2}). Thus, $\nabla_{\Phi(e)} \Phi(f)$ joins $[\psi(w) \sigma_{\alpha\circ d\psi^{-1}}]$ and
$[\psi(w) \sigma_{\alpha \circ d \psi^{-1}} \sigma_{\beta \circ d \psi^{-1}}]$, the same as $\Phi(\nabla_{e} f)$.
\end{proof}

Having fixed a signed structure through a choice of positive roots $\Delta_{+} \subset
\Delta_{G}$ (see the last paragraph of Section \ref{sec:homspaces})
by definition of Type 2 automorphisms we have:
\begin{lemma}\label{lem:type2signed}
The Type 2 automorphism associated to an automorphism $\psi:G\to G$ with $\psi(T)=T$ respects the signed structure on $\Gamma$ if and only if $\psi:\mft^\CC\to \mft^\CC$ permutes $\Delta_+$.
\end{lemma}

Recall that an inner automorphism of $G$ is by definition conjugation with an element of
$G$, and that the group of outer automorphisms of $G$ is defined as
${\mathrm{Out}}(G):=\Aut(G)/{\mathrm{Inn}}(G)$, the quotient of all automorphisms by the
inner automorphism group. We denote by $\Pi\subset \Delta_+$ the set of simple roots. Recall that as any
two maximal tori in $G$ are conjugate via $G$, and the Weyl group acts simply transitively
on the possible choices of positive roots, any automorphism of $G$ induces a well-defined
permutation of $\Pi$, i.e., an automorphism of the Dynkin diagram of $G$. Conversely, any
automorphism of the Dynkin diagram induces an automorphism $\phi:G\to G$ with $\phi(T)=T$,
such that $\phi:\mft\to \mft$ permutes $\Delta_+$; also this is well-known, but
follows also from Lemma \ref{LieAlgAuto} above. 
Thus, we have constructed a natural isomorphism $T_2^+(\Gamma)\cong {\mathrm{Out}}(G)$.

\begin{theorem}\label{thm:type decomposition}
Let $\mathrm{Aut}_{\nabla}(\Gamma)$ denote the group of automorphisms of $\Gamma$ which
preserve the canonical connection of $G/T$. We define an action of $T_{2}(\Gamma)$ on
$W(G)$ by
\[
	\rho \colon T_{2}(\Gamma) \longrightarrow \mathrm{Aut}(W(G)), \quad \rho(\Phi)[w] := \Phi[w].
\]
The map
\[
	\eta \colon W(G) \ltimes_{\rho} T_{2}(\Gamma) \longrightarrow \mathrm{Aut}_{\nabla}(\Gamma), \quad
	([w], \Phi) \longmapsto L_{[w]} \circ \Phi
\]
is an isomorphism, where the group on the left is the semidirect product induced by
$\rho$. 

Fixing a choice of positive roots $\Delta_+\subset \Delta_G$, and denoting by $\Aut_\nabla^+(\Gamma)$ the group of automorphisms of $\Gamma$ fixing the canonical connection of $G/T$ and the associated signed structure, then by restriction of $\eta$ we obtain an isomorphism $W(G)\ltimes_\rho T_2^+(\Gamma)\cong \Aut_\nabla^+(\Gamma)$.
\end{theorem}

\begin{proof}
Recall that $\rho$ is indeed a representation by definition of $T_{2}(\Gamma)$, cf.
Example \ref{ex:Type2}, and thus $W(G) \ltimes T_{2}(\Gamma)$ is a semidirect
product. To show that $\eta$ is a homomorphism, we have to show first the identity 
\[
	\Phi \circ L_{[w]} \circ \Phi^{-1}  = L_{\Phi[w]}.
\]
By Lemma \ref{AutoRigidity} and Lemma \ref{lem: type 1 and 2 preserve can connection} it suffices to check that both sides agree on a vertex and all
edges emerging from it. Clearly both sides agree on the identity element $[e]$ of $W(G)$.
Let $[\bar w]$ be a vertex different from $[e]$ which lies on an edge emerging from $[e]$.
Since $\Phi$ is an automorphism of $W(G)$ we have that both sides agree also on $[\bar
w]$. Since for $G/T$ the vertices $[e]$, $[\bar w]$ determine a unique edge, the identity
is proven. Now we have
\begin{align*}
	\eta(([w_{1}], \Phi_{1}) \cdot ([w_{2}], \Phi_{2})) &= \eta(
	([w_{1}]\rho(\Phi_{1})[w_{2}], \Phi_{1}\Phi_{2})) \\
																											&= L_{[w_{1}] \Phi_{1}[w_{2}]} \circ \Phi_{1} \circ \Phi_{2}\\
																											&= 	L_{[w_{1}]} \circ \Phi_{1} \circ
																											L_{[w_{2}]} \circ \Phi_{2}\\
																											&=\eta([w_{1}], \Phi_{1}) \circ
																											\eta([w_{2}], \Phi_{2}).
\end{align*}
It remains to show that $\eta$ is bijective. Suppose $\eta([w], \Phi) = \mathrm{id}$, Then
we have $[e] = L_{[w]}\Phi[e] = [w]$ and therefore $\Phi = \mathrm{id}$, which shows that
$\eta$ is injective. Now let $\Psi \in \mathrm{Aut}_{\nabla}(\Gamma)$ and set $[w]:=
\Psi[e]$ and $\Phi:= L_{[w]^{-1}} \circ \Psi$. The theorem is proven if  we show $\Phi \in
T_{2}(\Gamma)$. From $\Phi[e]=[e]$ we infer that
$\Phi$ permutes the edges emerging from $[e]$. 
From the definition of graph automorphisms, there
 is an isomorphism $\Phi_{\ast} \colon \mathfrak t^{\ast} \to \mathfrak t^{\ast}$ such
 that $\alpha(\Phi(e))) = \Phi_{\ast}(\alpha(e))$, where $\alpha$ is the axial function
 defined in Theorem \ref{thm:GHZ}. Thus $\Phi_{\ast}$ permutes the roots and Lemma
 \ref{LieAlgAuto} implies that $\Phi_{\ast}$ can be extended to an automorphism of
 $\mathfrak g$. Since $G$ is simply connected this automorphism corresponds to an
 automorphism $f_{\Phi} \colon G \to G$, such that $(f_{\Phi})_{\ast} = \Phi_{\ast} \colon
 \mathfrak g \to \mathfrak g$. The Type 2 automorphism induced from $f_{\Phi}$, see
 Example \ref{ex:Type2}, coincides  with $\Phi$, which follows from Lemma
 \ref{AutoRigidity}, because they coincide on all edges emerging from $[e]$.
 
 The statement on the automorphism group of the signed structure follows immediately from Lemma \ref{lem:type2signed} and the fact that Type 1 automorphisms respect it.
\end{proof}

	\section{Realizability}\label{sec:Realizability}

 Let $\pi\colon\Gamma'\to B$ be a $T$-GKM fiber bundle, where $B$ is a two-valent
 GKM-graph and fibers are isomorphic to the GKM graph $\Gamma$ of $G/T$ as described in Section
 \ref{sec:homspaces}. We assume now and throughout that any two weights at some vertex
 (and hence any vertex) of $B$, which are elements in the weight lattice of $T$, span a
 primitive lattice in the latter (or equivalently the common kernel of the weights of $B$ is
 a connected subgroup).
	We label the $n$ vertices of $B$ by $v_1,\hdots,v_n,v_{n+1}=v_1$,
	denote the edge connecting $v_i$ with $v_{i+1}$ by $e_i$, and we identify
	$\Gamma$ with $\Gamma_{v_1}'$, the fiber over $v_1$, via a fixed isomorphism. By definition of a GKM fibration,
we obtain isomorphisms of GKM graphs $\Phi_{e_i}:\Gamma'_{v_i}\to \Gamma'_{v_{i+1}}$,
whose concatenation defines an automorphism $\Phi:=\Phi_{e_n}\circ \hdots \circ
\Phi_{e_1}\colon \Gamma \to \Gamma$. We denote the associated linear maps by $\Psi_{e_i}$
and put $\Psi:= \Psi_{e_n}\circ\hdots\circ \Psi_{e_1}$. In the following, we will refer to
$\Phi$ as the \emph{twist automorphism} of the GKM fiber bundle $\pi$.

Let us fix on $\Gamma$ a signed structure associated to some choice of positive roots $\Delta_+\subset \Delta_G$.

\begin{lem}\label{lem: fiberwise signed}
If the twist automorphism $\Phi$ respects the signed structure, i.e., is an element of $\Aut_\nabla^+(\Gamma)$, then the fiber bundle carries a fiberwise signed structure.
\end{lem}

\begin{proof}
The fiber graph $\Gamma = \Gamma'_{v_1}$ carries a signed structure. The isomorphisms
$\Phi_{e_i}$ inductively induce signed structures on all other fibers such that by
definition $\Phi_{e_i}$ preserves the signed structures on $\Gamma_{v_i}'$ and
$\Gamma_{v_{i+1}}$. It remains to check that this is globally well defined, i.e.\ after a
full rotation around $B$, $\Phi_{e_n}$ is compatible with the chosen signed structures on
$\Gamma_{v_n}'$ and $\Gamma_{v_1}'$. But this is equivalent to $\Phi$ preserving the signed
structure on $\Gamma_{v_1}'$.
\end{proof}
	
	The goal of this section is the proof of the following

	\begin{theorem}\label{thm:realization}
		Let $\Gamma\to \Gamma'\to B$ be a $T$-GKM fiber bundle, where
		$\Gamma$ is the GKM graph of $G/T$ and $B$ is $2$-regular.
		Furthermore assume that the common kernel of the weights of $B$ is
		a connected subgroup of $T$ and that the twist automorphism $\Phi$ lies in
	$\Aut_\nabla^+(\Gamma)$. Denote by $\Phi={\Phi}_1\circ{\Phi}_2$ the decomposition as in
	Theorem \ref{thm:type decomposition}, with ${\Phi}_1 = L_{[w]}$ for some $w\in W(G)$ and
	${\Phi}_2 \in T_2^+(\Gamma)$.
		\begin{enumerate}
			\item If ${\Phi}_2=\Id$, then the bundle is realizable. More precisely, there exists a smooth $T$-equivariant fiber bundle $Z\to X$ with fibers over fixed points twisted equivariantly diffeomorphic to $G/T$, such that the $T$-action on $Z$ is of GKM type with GKM graph $\Gamma'$ and $X$ is a quasitoric $4$-fold with GKM graph $B$.
			\item If ${\Phi}_2\neq \Id$, then the map $H^*(\Gamma')\rightarrow H^*(\Gamma)$ induced by the fiber inclusion is not surjective. In particular, the GKM fiber bundle is not realizable by an equivariant fiber bundle in which base, total space, and fibers over fixed points are GKM manifolds.
		\end{enumerate}	
		
	\end{theorem}
	
\begin{remark}
A version of this theorem in which the assumption that the twist automorphism respects the
signed structure was missing is contained in \cite{Wardenski2023}.
\end{remark}	
	
\begin{ex}
Consider the Lie group $G=(\SU(2))^3$ with maximal torus $T=T^3$. Let $x,y,z$ denote the weights dual to the circles acting on each factor. Then $G/T\cong (S^2)^3$ and the GKM graph is a cube with edge labels $x,y,z$, where parallel edges share the same weight. As base space we take $\C P^2$ with the weights

\begin{center}
\begin{tikzpicture}

\draw[very thick] (0,0) -- ++(2,0) -- ++(-2,2) -- ++(0,-2);

\node at (1,-0.4){$x$};
\node at (1.5,1.3){$x-y$};
\node at (-0.4,1){$y$};

 \node at (0,0)[circle,fill,inner sep=2pt]{};
 \node at (2,0)[circle,fill,inner sep=2pt]{};
 \node at (0,2)[circle,fill,inner sep=2pt]{};

\end{tikzpicture}
\end{center}

Over this we consider the GKM fibre bundle given by

\begin{center}
	\begin{tikzpicture}

\draw[very thick] (0,0) -- ++(4,0) -- ++(4,0) -- ++(4,0);

\draw[very thick] (0,0) -- ++(-1.2,1.5);
\draw[very thick] (0,0) -- ++(0,2);
\draw[very thick] (0,0) -- ++(1.2,1.5);

\node at (-1.1,1.8) {$x+z$};
\node at (0,2.3) {$z$};
\node at (1.1,1.8) {$y+z$};

\draw[very thick] (4,0) -- ++(-1.2,1.5);
\draw[very thick] (4,0) -- ++(0,2);
\draw[very thick] (4,0) -- ++(1.2,1.5);

\node at (2.9,1.8) {$x+z$};
\node at (4,2.3) {$z$};
\node at (5.1,1.8) {$x+y+z$};

\draw[very thick] (8,0) -- ++(-1.2,1.5);
\draw[very thick] (8,0) -- ++(0,2);
\draw[very thick] (8,0) -- ++(1.2,1.5);

\node at (6.9,1.8) {$y+z$};
\node at (8,2.3) {$z$};
\node at (9.1,1.8) {$x+y+z$};

\draw[very thick, dashed] (12,0) -- ++(-1.2,1.5);
\draw[very thick, dashed] (12,0) -- ++(0,2);
\draw[very thick, dashed] (12,0) -- ++(1.2,1.5);

\node at (10.9,1.8) {$y+z$};
\node at (12,2.3) {$z$};
\node at (13.1,1.8) {$x+z$};

\node at (2,-0.5) {$x$};
\node at (6,-0.5) {$x-y$};
\node at (10,-0.5) {$y$};

 \node at (0,0)[circle,fill,inner sep=2pt]{};
 \node at (4,0)[circle,fill,inner sep=2pt]{};
 \node at (8,0)[circle,fill,inner sep=2pt]{};
  \node at (12,0)[circle, draw = black, fill = white, inner sep=2pt]{};

\end{tikzpicture}
\end{center}
where horizontal edges are in fact horizontal, fibers are the cubes spanned by the three
non-horizontal edges at every vertex (where horizontal edges of the respective label
emanate from every vertex of the cube, only one of which is visible in the picture), and
the rightmost depicted vertex is glued to the first vertex, with edges of same label being
identified. In particular the twist automorphism swaps the two edges with label $x+z$ and $y+z$ in the left most fiber.  Note that
this indeed defines a GKM fiber bundle: as it preserves the canonical connection, we
obtain an induced connection on the total space of the fibration satisfying the conditions
in Definitions \ref{defn:GKMfibration} and \ref{defn:GKMfiberbundle}. The twist
automorphism is of Type 2, induced by the automorphism of $G$, swapping the first and
third factor, hence by Theorem \ref{thm:realization} this GKM fiber bundle cannot be
realized geometrically.

A somewhat systematic construction of realizable bundles is discussed in Section \ref{sec: examples}.
\end{ex}

	\begin{proof}[Proof of $(i)$]
 We use the notation introduced in the beginning of the section.
 We claim that the dual of the linear map $\Psi_{e_i}:\mft^*\to \mft^*$ fixes the
codimension $1$-subgroup $K_i$ defined by the weight $\alpha(e_i)$ of $e_i$. To see this,
consider any weight $\beta$ of a vertical edge in the fiber over $v_i$. Then, for any
$v\in \ker \alpha(e_i)$ we have $\beta(\Psi_{e_i}^*(v)) = \Psi_{e_i}(\beta)(v) = (\beta +
c\alpha(e_i))(v) = \beta(v)$ for some constant $c$; note that there do not appear any
signs because the fiber bundle is fiberwise signed. Hence, as the fiber weights span
$\mft^*$, it follows that $\Psi_{e_i}^*(v) = v$. Thus, the common kernel of the weights of
$B$ gives a codimension $2$	subtorus $T'\subset T$ which is fixed by all the
$\Psi_{e_i}^*$. In particular, the automorphism $\Psi$ associated to the twist
automorphism $\Phi$, which is in $(i)$ assumed to be of Type 1, is of the form
$\Psi=L_{[w]}$, with $w$ contained in the (connected) centralizer $Z_T(T')$.

As explained in Section \ref{sec:quasitoric}, the GKM graph $B$ as in
Theorem \ref{thm:realization} is realizable by a non-effective strongly quasitoric
manifold $X$ with GKM graph $B$. Its orbit space is diffeomorphic, as a manifold with
corners, to a convex $2$-dimensional polytope $P$, whose boundary corresponds
combinatorially to the graph $B$. The kernel of the action is the codimension $2$
subtorus $T'$. There is a continuous section of the orbit map projection $\pi\colon
X\rightarrow P$ which induces, as explained in Section \ref{sec:quasitoric}, an
equivariant homeomorphism $X=P\times T/_\sim$, where $\sim$ identifies points in the
$T$-fibers in the following way: over	the interior of $P$ we project to $T/T'$, over each
edge in the boundary we divide by the codimension $1$ subtorus which is the kernel of the
corresponding weight in $B$ and over the vertices we collapse the entire fiber.	We choose
this section, and hence the identification $X\cong P\times T/_\sim$, as in Lemma \ref{lem:
quasi smooth}. Let $D\subset P$ be a small open disk in the interior. 
		\begin{figure}[h]
			\centering
			\includegraphics{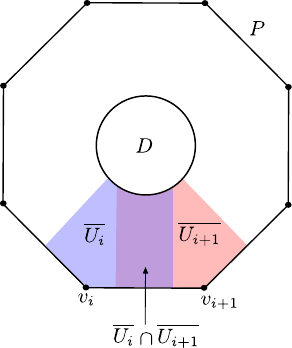}
		\end{figure}

		Now we divide $P\backslash D$ into open regions
		$\overline{U_i}$ where $\overline{U_i}$ contains a single vertex $v_i$ and
		$\overline{U_{i}}\cap \overline{U_{i+1}}$ intersect as depicted the figure above (note
		that parts of the boundary of $D$ and parts of the edges are contained in
		$\overline{U_{i}}$). Let
		$U_i=\pi^{-1}(\overline{U_i})$.
		
		We first construct a certain principal $G$-bundle. For $i=1,\ldots,n+1$, let
		$\widetilde{\Psi}_i\colon T\rightarrow T$ be the automorphism
		induced by $\Psi_{e_i}^*$ and $G_i$ denote the $T$-space which is $G$ with the action
		given by left multiplication precomposed with the automorphism
		$\widetilde{\Psi}_{i-1}\circ\ldots\circ\widetilde{\Psi}_1$.
		We consider $B_i= U_i\times G_i$ with the diagonal action (we treat $B_1$ and $B_{n+1}$
		as different $T$-spaces even though the underlying spaces are the same). Let $S_i=U_i\cap
		U_{i+1}\cong (0,1) \times [0,1]\times T/_\sim$, where the $[0,1]$ component runs from $\partial P$ at $0$ to $\partial D$ at $1$ and $\sim$ collapses $T$ to $T/T'$ over points in
		$(0,1) \times (0,1]$ and to $T/K_i$ over $(0,1) \times \{0\}$. By construction, we have
		that $S_{i}$ is a smooth submanifold of $X$ with non-empty boundary (given by the
		preimage of a part of the boundary of $D$ under $\pi$).  For $i=1,\ldots,n$, we glue $B_i$ and $B_{i+1}$
		along equivariant diffeomorphism
		$S_i\times G_i\rightarrow S_i\times G_{i+1}$ defined as follows: consider the unique $T$-equivariant diffeomorphism
		\[
		(0,1) \times [0,1]\times T\times G_i\longrightarrow (0,1) \times [0,1]\times T\times G_{i+1}
		\] 
		(with $T$ acting diagonally on $T\times G_i$) which restricts to the identity on $(0,1)\times [0,1]\times\{1\}\times G$.
		It is of the form $(a,b,t,g)\mapsto (a,b,t,\varphi_t(g))$ with $\varphi_t(g)=t\cdot (t^{-1}\cdot g)$, where the left hand multiplication is understood as the action of $G_{i+1}$ and the right hand multiplication is in $G_i$.
		Now
		since the $K_i$-actions on $G_i$ and $G_{i+1}$ agree by construction it follows that
		$t\mapsto \varphi_t$ descends to $T/K_i$ and hence we get an induced $T$-equivariant
		homeomorphism 
		\[S_i\times G_i\longrightarrow
		S_i\times G_{i+1};\, .\]

		Smoothness can be checked componentwise with the only
		interesting case being that the map $S_i\times G_i\rightarrow G_{i+1}$ is
		smooth. Let $p\colon S_i\rightarrow T/K_i$ be the projection. Then this factors as
		$(s,g)\mapsto \varphi_{p(s)}(g)$. Clearly $\varphi_t(g)$ depends smoothly on $t$ and
		$g$, and $p$ is smooth by Lemma \ref{lem: quasi smooth}.

		By assumption the automorphism $\widetilde \Phi_n\circ \ldots \circ \widetilde \Phi_1$ is
		given by conjugation $c_w$ with $w\in N_G(T)$. The map $(x,g)\mapsto (x,w^{-1}g)$ is hence
		a $T$-equivariant diffeomorphism $B_{n+1}\rightarrow B_1$ which we use to identify the
		two. Let $Y$ denote the resulting space given by the gluing of all the $B_i$. It
		carries a left $T$-action commuting with the obvious right $G$-action acting on all
		the $G_i$-factors by right multiplication. The latter gives $Y$ the
		structure of a principal $G$-bundle over $X\backslash \pi^{-1}(D)$.
		
		We investigate this bundle over the boundary $\partial D\cong S^1$. We claim that
		$Y|_{\partial D}$ is $(T \times G)$-equivariantly diffeomorphic to the following mapping torus: Consider the unique $T$-equivariant extension $T\times G_1\rightarrow T\times G_1 $ of $(1,g)\mapsto (1,w^{-1}g)$. Since $c_w$ fixes $T'$ it follows that $w\in Z_G(T')$ and hence the above map descends to a diffeomorphism \[f \colon T/T'\times G_1\longrightarrow T/T'\times G_1.\]
		We note that $f$ is equivariant with respect to the left $T$-action and the right $G$-action and hence the mapping torus $M_f = [0,1]\times T/T'\times G_1/\sim$ is a smooth $(T\times G)$-manifold.
		
		Now in order to see $Y|_{\partial D} \cong M_f$ we observe that $B_i|_{\partial D}$ is given by $J_i\times T/T'\times G_i$ for some interval $J_i\subset \partial D$. For an appropriate parametrization $\partial D\cong [0,1]/_\sim$, the unique $T$-equivariant maps $J_i\times T/T'\times G_1\rightarrow B_i$ which are the identity on $J_i\times \{1\cdot T'\}\times G$ piece together to a map from $[0,1]\times T/T'\times G_1$ into the gluing of $B_1$ up until $B_{n+1}$. After identifying $B_{n+1}$ and $B_1$ as above this descends to the desired diffeomorphism $M_f\rightarrow Y|_{\partial D}$.
		
		Since $Z_G(T')$ is connected (as $T'$ is connected), $f$ is $(T\times G)$-equivariantly isotopic to the identity. Consequently $Y|_{\partial D}$ is in fact $(T\times G)$-equivariantly diffeomorphic
		to $S^1\times T/T'\times G_1$. Hence we may complete $Y$ to a closed smooth $(T\times G)$-manifold
		$Z$ by gluing in $D\times T/T'\times G$. By construction, the left $T$-action on the
		quotient of $Z$ by the right $T$-action coming from the right $G$-action gives the desired
		$T$-equivariant $G/T$-bundle over $X$.
	\end{proof}
	
	For the proof of $(ii)$ we make some preliminary observations.
	\begin{lem}\label{lem: cohomologystuff}
		Let $\Gamma$ be the GKM graph of $G/T$.
		\begin{enumerate}[label=(\alph*)]
			\item Any automorphism of Type 1 induces the identity on $H^*(\Gamma)$.
			\item Any nontrivial automorphism of Type 2 induces a nontrivial map on $H^2(\Gamma)$.
		\end{enumerate}
	\end{lem}
	
	\begin{proof}
		Since $H^*(\Gamma)\cong H^*(G/T)$ and both types of automorphisms correspond to
		continuous transformations of $G/T$ it is sufficient to work with the latter. Left
		multiplication with elements in $W_G(T)$ is non-equivariantly homotopic to the identity
		since $G$ is connected. This proves $(a)$. For part $(b)$ we may consider a Type 2
		automorphism of $\Gamma$ which is induced by an automorphism $\varphi\colon
		G\rightarrow G$ preserving $T$ and inducing a nontrivial map on $\mathfrak{t}^*$ (as
		otherwise the automorphism of the graph is the identity by Lemma \ref{AutoRigidity}).
		We consider $G$ as a $T$-space with the standard $T$-action by left
		multiplication. Using a functorial construction of the universal $T$-bundles
		(see \cite{MR77122}) $\varphi$ induces an automorphism of the universal
		$T$-bundle i.e.\ a map $\widetilde \varphi\colon ET\rightarrow ET$ such that
		$\widetilde \varphi(t\cdot e)= \varphi(t)\cdot \widetilde \varphi(e)$. Now
		$(\varphi,\widetilde \varphi)\colon G\times ET \rightarrow G\times ET$ is twisted
		equivariant with respect to $\varphi|_T$ and induces an automorphism of the Borel
		construction $G_T$. Thus we obtain a commutative diagram
		\[\xymatrix{
			G/T\ar[d] & G_T\ar[d]\ar[r]\ar[l] & BT\ar[d]\\
			G/T & G_T \ar[r]\ar[l]& BT
		}\]
		where all vertical maps are the ones induced by $\varphi$ and the horizontal maps are the standard projections of the Borel construction $G_T$. All horizontal arrows induce isomorphisms on $H^2$: the left hand arrows are homotopy equivalences due to the freeness of the $T$-action and for the right hand side this follows from the Serre spectral sequence of the Borel fibration due to $G$ being simply connected. Hence it suffices to show that $H^2(BT)\rightarrow H^2(BT)$ is nontrivial. The transgression of the Serre spectral sequence of the universal $T$-bundle identifies this with the map $\varphi^*\colon H^1(T)\rightarrow H^1(T)$, which, after passing to real coefficients, can be identified with the map induced on $\mft^*$. In particular it is nontrivial.
	\end{proof}
	
	\begin{lem}\label{lem: H2lem}
		Let $\Gamma$ be any GKM graph. For any unoriented edge $e\in E(\Gamma)$, let $S_e$ be the graph consisting of $e$ and its adjacent vertices. Then the map
		\[H^2(\Gamma)\longrightarrow \bigoplus_{e\in E(\Gamma)} H^2\left( S_e\right)\]
		induced by the inclusions is an injection.
	\end{lem}
	
	\begin{proof}
		If $x\in H_T^2(\Gamma)$ projects to a nontrivial element of $H^2(\Gamma)$, then there is an edge $e\in E(\Gamma)$ where $x$ takes different values at the vertices adjacent to $e$. In particular the image of $x$ in $H^2(S_e)$ is nontrivial.
	\end{proof}

	\begin{proof}[Proof of Theorem \ref{thm:realization} part $(ii)$.]
		Note first that if the bundle were realizable, then the Serre spectral sequence of the fiber bundle would collapse since the cohomology of GKM manifolds is concentrated in even degrees. This is however equivalent to cohomological surjectivity of the fiber inclusion. For the proof of the combinatorial statement we assume that the inclusion $\iota\colon \Gamma\rightarrow \Gamma'$ induces a surjection on cohomology and show that $\Phi$ induces the trivial map on $H^2(\Gamma)$. Then it will follow from Lemma \ref{lem: cohomologystuff} that $\Phi_2=\Id$.

		For $e\in E(\Gamma)$ let $i_e\colon S_e\rightarrow \Gamma$ be the inclusion of the GKM
		subgraph consisting of $e$ and its adjacent vertices. By Lemma \ref{lem: H2lem}, in
		order to show $\Phi^*=\Id$ it suffices to show $(\Phi\circ i_e)^*=i_e^*$ for all $e\in
		E(\Gamma)$. Due to the cohomological surjectivity of $\iota$ it is in fact sufficient
		to prove $(\iota\circ\Phi\circ i_e)^*=(\iota\circ i_e)^*$. 
		
		Recall
		that the twist morphism $\Phi$ is of the form $\Phi=\Phi_{e_n}\circ\hdots\circ
	\Phi_{e_1}$, where $\Phi_{e_i}:\Gamma'_{v_i}\to \Gamma'_{v_{i+1}}$ is an isomorphism of
GKM graphs with respect to the linear isomorphisms $\Psi_{e_i}:\mft^*\to \mft^*$.

Denote by $w_{11}$ and $w_{12}$ the two vertices of the vertical edge $e$ in $\Gamma =
\Gamma'_{v_1}$; applying the $\Phi_{e_k}$ we obtain vertices $w_{i j}=\Phi_{e_{i-1}}\circ
\cdots \circ \Phi_{e_1}(w_{1j})$, for $j=1,2$. The vertices $w_{i1}$ and $w_{i2}$ are
connected by the edge $\Phi_{e_{i-1}}\circ \cdots \circ \Phi_{e_1}(e)$, with label
$\gamma_i:=\Psi_{e_{i-1}}\circ \cdots \circ \Psi_{e_1}(\alpha(e))$. We have to show that
for any cohomology class $x\in H^2_T(\Gamma')$ the pair consisting of the evaluations
$(x(w_{11}),x(w_{12}))$ differs from the pair $(\Psi_{e_1}^{-1}\circ \cdots \circ
\Psi_{e_n}^{-1}(x(w_{n1})),\Psi_{e_1}^{-1}\circ \cdots \circ \Psi_{e_n}^{-1}(x(w_{n2})))$
only by an element of the form $(\beta,\beta)$, with $\beta\in
H^2(BT;\ZZ)^*$, implying that the images of
$(\iota\circ i_e)^*(x)$ and $(\iota\circ \Phi\circ i_e)^*(x)$ agree in $H^2(S_e)$. 

We will show this statement for the pair $(\Psi_{e_i}(x(w_{i1})),\Psi_{e_i}(x(w_{i2})))$ and $(x(w_{i+1\, 1}),x(w_{i+1\, 2}))$, which inductively proves the claim above.

From Lemma \ref{lem: fiberwise signed} we obtain a fiberwise signed structure, such that all vertical weights have fixed signs
compatible with the $\Phi_{e_i}$. For horizontal weights we fix an arbitrary sign. We
then have the following relations:
\begin{align*}
x(w_{i+1\, 1}) - x(w_{i 1}) &= k \alpha(e_i)\\
x(w_{i+1\, 2}) - x(w_{i 2}) &= l \alpha(e_i)\\
x(w_{i 2}) - x(w_{i 1}) &= m \gamma_i\\
x(w_{i+1\, 2}) - x(w_{i+1\, 1}) &= n \gamma_{i+1}
\end{align*}
for some $k,l,m,n\in \ZZ$. We have to compare
\begin{align*}
x(w_{i+1\, 1}) - \Psi_{e_i}(x(w_{i1})) &= x(w_{i1}) - \Psi_{e_1}(x(w_{i1})) + k\alpha(e_i)
\end{align*}
with
\begin{align*}
x(w_{i+1\, 2}) - \Psi_{e_i}(x(w_{i 2})) &= x(w_{i2}) + l\alpha(e_i) - \Psi_{e_i}(x(w_{i1}) + m\gamma_i)\\
&= x(w_{i1}) - \Psi_{e_i}(x(w_{i1})) + l\alpha(e_i) + m\gamma_i - m\Psi_{e_i}(\gamma_i).
\end{align*}
Due to the fiberwise signed condition we have positive signs in the equation $\Psi_{e_1}(\gamma_i) = \gamma_i + u \alpha(e_i)$  for some $u\in \ZZ$ whence we have to show that
\begin{align*}
k = l-mu.
\end{align*}
But combining the equations above, we have
\begin{align*}
n(\gamma_i + u\alpha(e_i)) &= n\Psi_{e_i}(\gamma_i) = n\gamma_{i+1} = x(w_{i+1\, 2}) - x(w_{i+1\, 1}) \\
&= x(w_{i2})+ l\alpha(e_i) - x(w_{i1}) - k\alpha(e_i)\\
&= m\gamma_i + (l-k)\alpha(e_i),
\end{align*}
hence by the GKM condition we have $n=m$ and $l-k = nu = mu$.
	\end{proof}
	
	\section{Examples of GKM fiber bundles}\label{sec: examples}
	\subsection{Admissible tuples}
Finally, we would like to give a large class of examples of $T$-GKM fiber bundles $\Gamma\to \Gamma'\to B$, where $\Gamma$ is the GKM graph of $G/T$ and $B$ a $2$-regular GKM graph such that the common kernel $T'$ of the weights of $B$ is connected, whose twist automorphism is of \textbf{Type 1}, see Example \ref{Type1}. These GKM fiber bundles will be realizable by Theorem \ref{thm:realization}. We will try to be as general as possible at first, in order to make the class of examples as large as possible, but will certainly have to make some restrictions later on.

Choose a compact simply-connected Lie group $G$, let $T\subset G$ be a maximal torus, and
$w\in W(G)$ such that $c_w$ fixes a subtorus $T'\subset T$ of codimension $2$.
Let $B$ be any two-valent
GKM graph whose labels are elements of the annihilator
${\mathrm{ann}}(\mft')\subset \mft^*$. We assume that the common kernel of the labels in
${\mathrm{ann}}(\mft')\subset \mft^*$ is equal to the connected subgroup $T'$ --
among others, this allows us to apply Theorem \ref{thm:realization}.

Having fixed $w$, $T'$ and $B$ we ask how to construct signed GKM fiber
bundles $\Gamma \to \Gamma'\to B$ with twist automorphism of Type 1
defined by $w$, i.e., $\beta\mapsto w\cdot \beta = \beta\circ \Ad_w^{-1}$, and subject
to the condition that the labels of the fiber graph over some base vertex $v_1$ are
the standard labels of the GKM graph of $G/T$. Starting there, we
label the $n$ vertices of $B$ by $v_1,v_2,\hdots,v_n,v_{n+1}=v_1$, denote by $\alpha_i$
the signed weight of the edge $(v_i,v_{i+1})$, and by $T_i$ the subgroup given by the
kernel of $\alpha_i$.\\

To construct the total space of $\Gamma \to \Gamma'\to B$, we consider $\Gamma' :=
(\gamma\times \Gamma)/\sim$, where $\gamma$ is an abstract path graph with $n+1$ vertices
$v_i$ and fibers over $v_1$ and $v_{n+1}$ are identified via the twist automorphism
$L_{[w]}$ of
Type 1, see Example \ref{Type1}. So we only need to care about the
labeling and the connection. We equip the horizontal edges with the
labels corresponding to a simple closed path around $B$. It remains to specify what the
labeling on all vertical edges is.\\

Assume for the moment that we have already fixed the labeling on the fiber $\Gamma_{v_1}'$ over $v_1$, endowed with
its canonical connection. Now we ask what the edge weights of $\Gamma_{v_2}'$ might look
like. From the definition of a GKM fiber bundle, see Section
\ref{sec:GKMfiberbundles}, we have that any potential set of edge labels of
$\Gamma_{v_2}'$ arises from $\Gamma_{v_1}'$ by identifying the graphs via horizontal
transport and applying to the labels an element in $\mathcal{A}_1$, the group of
automorphisms of $T$ that fix $T_1=\ker(\alpha_1)$. Conversely, any such automorphism
defines in this way a set of edge labels on $\Gamma_{v_2}'$.\\

We can go on like this, defining in particular the groups $\mathcal{A}_{i}$, until we
reach $\Gamma_{v_{n}}'$. We have now chosen elements
\[
\psi_1\in \mathcal{A}_1, \hdots, \psi_{n-1}\in \mathcal{A}_{n-1}.
\]
In principle, this defines the labels on the entire graph $\Gamma$. In
order for the result to define a GKM fiber bundle, the labels of $\Gamma_{v_n}'$ need to
be compatible with $\Gamma_{v_1}'$ in the same way, i.e., $\psi_n\circ \hdots \circ \psi_1=\Ad_w$. This leads us to the following
\begin{definition}
	We call a tuple $(\psi_1,\hdots,\psi_n)$ \textit{admissible} if $\psi_n\circ \hdots \circ \psi_1=\Ad_w$, and define $\mathcal{A}_{adm} \subset \mathcal{A}_1\times \hdots \times \mathcal{A}_n$ to be the set of admissible tuples.
\end{definition}
In order to construct examples of GKM fiber bundles, we therefore have to
\begin{itemize}
	\item construct admissible tuples.
	\item Give a criterion for when the labeling obtained by such tuples actually
	satisfies the GKM condition of two adjacent weights being linearly
	independent. Note that this condition has not yet been considered in the
	constructions of the labels and that in theory fiber weights might become colinear to
	adjacent basic weights in the construction process.
	\item find a compatible connection with respect to the now obtained labeling.
\end{itemize}
Let us first treat the last point. On $\gamma\times \Gamma$ we consider the connection of
product type, restricting to the canonical connection on each fiber. This induces a
well-defined connection on the quotient $\Gamma'$, since the canonical connection is
preserved by any automorphism of Type 1, see Lemma \ref{lem: type 1 and 2 preserve can
connection}. The resulting connections is compatible with the labels as transport along a
horizontal edge of label $\beta$ shifts vertical labels by an automorphism which restricts
to the identity on $\Z^n/\langle \beta\rangle $. \\

\begin{remark}\label{rem:admissible}
	Note that the labeling of $\Gamma_{v_1}'$ does not enter in the definition of
	'admissibility'. Note further that the GKM condition (ii) in Definition
	\ref{defn:abstractgkmgraph} is satisfied at each edge. Thus, given an admissible tuple,
	we may hope to choose the initial labeling $\Gamma_{v_1}'$ in such a way that the induced
	labels at each vertex of $\Gamma'$ are pairwise linearly independent. This is achieved by the following lemma.
\end{remark}
	\begin{lemma}
		Fix $\Gamma_{v_1}'$, $B$, $\Phi$, and an admissible tuple $(\psi_1,\hdots,\psi_n)$ as above.
		\begin{enumerate}[label=$(\roman*)$]
			\item 
		 If none of the labels on $\Gamma_{v_1}'$ are contained in the real plane spanned by the base weights, then the resulting labeled graph $\Gamma'$ is a GKM graph.
		 \item One can always choose labels on $\Gamma_{v_1}'$ subject to the condition in
			 $(i)$ while fixing the other choices in the construction.
		 \end{enumerate}
	\end{lemma}
	\begin{proof}
		We start with the proof of statement $(ii)$. Let $\gamma_1,\hdots,\gamma_k$ be the set
		of all of the weights of $\Gamma\cong \Gamma_{v_1}'$ with its standard labeling, see
		Section \ref{sec:homspaces}. We identify the weight lattice with $\ZZ^m$.
		Choose a rational plane in $\Q^m$ not containing the $\gamma_i$ and
		we define $E\subset \ZZ^m$ as its integral points. There is an automorphism $\psi$ of
	$\ZZ^m$ that sends $E$ to $\mathrm{ann}(\mft')\subset \mft^*$, where the codimension two
subtorus $T'$ is fixed by $\Ad_w$. Clearly none of the $\psi(\gamma_i)$
is contained in $\mathrm{ann}(\mft')$. The labels on $\Gamma'_{v_1}\cong \Gamma$ given by
the $\psi(\gamma_i)$ do the job. 
		
		In order to show that in this case $\Gamma$ is a GKM graph, we only need to show that
		the labels at any vertex of $\Gamma'$ are pairwise linearly independent. This can be
		seen as follows: at each vertex, it is already clear that no pair of
		horizontal and no pair of vertical weights are colinear. Also, for a
		vertex above $v_1$, no base weight is colinear to a vertical weight, because by
		assumption none of the vertical weights is contained in the real span of the base
		weights.
		
		We observe that the last condition holds for each vertex over $v_2$, because it holds
		over $v_1$ and vertical weights over $v_2$ are the same as vertical weights over $v_1$
		modulo base weights (due to GKM condition $(ii)$). Inductively, it holds over all
		vertices of $B$. This finishes the proof.
	\end{proof}
\subsection{The construction of admissible tuples}
To construct admissible tuples, it is more convenient to identify automorphisms of
$T$ with their respective automorphisms of the Lie algebra, that is, with elements of
$\GL(m,\ZZ)$. In that way, each $\psi_i\in \mathcal{A}_i$ is now an element in
$\GL(m,\ZZ)$ which fixes $\mft_i$ (the Lie algebra of $T_i$) and in particular $\mft'\subset
\mft$.
We choose a basis $w_1,\hdots,w_{m-2}$ for the canonical lattice of the latter, given by the kernel of the exponential map, and extend this to a basis $w_1,\hdots,w_m$ for the lattice of $\mft$ such that
$w_{m-1}$ is in $\mft_{n}$ and $w_m$ is in $\mft_{n-1}$. 
To see that this is possible, consider the exact sequence of tori $T'\to
T\to T^2$ defined by the homomorphism $(\alpha_{n-1},\alpha_n):\mft\to \ZZ^2$ which splits
by \cite[p. 57, Exercise 7]{MR336651} (here it enters that we assumed the common kernel of
$\alpha_{n-1}$ and $\alpha_n$ to be equal to the connected group $T'$). Then we define $w_{m-1}$ and $w_m$ to be elements in $\mft$ that are sent to $(0,1)$ and $(1,0)$ respectively.\\
With respect to this basis, any $\psi_i\in \mathcal{A}_i$ is of the form
\[
\begin{pmatrix}
	1_{\ZZ^{m-2}} & B_i \\
	0 & A_i
\end{pmatrix}
\]
where $B_i$ is some $(m-2)\times 2$-matrix and $A_i$ is in ${\GL}(2,\ZZ)$, such that $A_i$
and $B_i$ are $T_i$-\textbf{compatible} in the sense that the combined
matrix fixes $\mft_i$.
\begin{rem} \label{rem: Ti-compatible} An equivalent formulation is the following: the lattice of $T_i$ has a basis consisting of $w_1,\ldots,w_{m-2}$ and $v=a w_{m-1}+ b w_m$ for $a,b\in \ZZ$. Then for $v= (a,b)^T$ the condition of $T_i$-compatibility becomes $A_iv=v$ and $B_iv=0$.
\end{rem}
Note that by the choice of $w_{m-1},w_m$, the matrices $A_{n-1}$ and $A_n$ need to be of the form
\[
A_{n-1}=\begin{pmatrix}
	{\pm 1} & 0 \\
	* & 1
\end{pmatrix}
,\quad
A_{n}=\begin{pmatrix}
	1 & * \\
	0 & {\pm 1}
\end{pmatrix}
\]
Of course, $\Ad_w$ has the same form as the $\psi_i$, and is thus determined by matrices $B_w$ and $A_w$.
\begin{theorem}\label{thm:fiberbundles}
	For the basis $(w_1,\hdots,w_m)$, any choice of the above $A_1,\hdots,A_{n}$ such that $A_{n}\cdot \hdots \cdot A_1=A_w$, and for any choice of compatible $B_1,\hdots,B_{n-2}$, there are unique compatible $B_{n-1}$ and $B_n$
	such that $\psi_n\cdot \hdots \cdot \psi_1=\Ad_w$, where $\psi_{i}$ is the automorphism defined by the matrix
	\[
		\begin{pmatrix}
			1_{\mathbb{Z}^{m-2}} & B_{i}\\
			0 & A_{i}\\
		\end{pmatrix}. 
	\]
\end{theorem}
\begin{proof}
	We set $\psi=\Ad_w\cdot \psi_1^{-1}\cdot \hdots \cdot \psi_{n-2}^{-1}$. This is of the form
	\[
	\psi=
	\begin{pmatrix}
		1_{\ZZ^{m-2}} & B \\
		0 & A_n\cdot A_{n-1}
	\end{pmatrix}.
	\]
	Here, $B$ is of the form $(u,v)$, where $u,v \in \mathbb{Z}^{m-2}$. We wish to decompose this matrix $\psi$ as
	$\psi_n\cdot\psi_{n-1}$. Recall that $A_{n-1}$ and $A_n$ are given by
	\[
	A_{n-1}=
	\begin{pmatrix}
		\pm1 & 0 \\
		k_{n-1} & 1
	\end{pmatrix}, \quad
	A_{n}=
	\begin{pmatrix}
		1 & k_n \\
		0 & \pm 1
	\end{pmatrix}.
	\]
	for certain integers $k_{n-1}$ and $k_n$.
	Moreover, due to compatibility, the right entries of $B_{n-1}$ as well as the left
	entries of $B_n$ have to be $0$, so we make the ansatz
	\[
	\psi_{n}=
	\begin{pmatrix}
		1_{\ZZ^{m-2}} & (0,v') \\
		0 & A_{n}
	\end{pmatrix},\quad
	\psi_{n-1}=
	\begin{pmatrix}
		1_{\ZZ^{m-2}} & (u',0) \\
		0 & A_{n-1}
	\end{pmatrix}
	\]
	which satisfy the compatibility conditions since by the initial assumptions we have $\mft_{n-1}=\langle w_1,\ldots,w_{m-2},w_m\rangle$ and $\mft_n=\langle w_1,\ldots,w_{m-1}\rangle$.
	Their product $\psi_n\cdot \psi_{n-1}$ is
	\[
	\psi_{n}\cdot \psi_{n-1}=
	\begin{pmatrix}
		1_{\ZZ^{m-2}} & (u'+k_{n-1}v', v') \\
		0 & A_{n}\cdot A_{n-1}
	\end{pmatrix}
	\]		
	Thus, $u'=u-k_{n-1}v$ and $v'=v$ is the unique solution. This finishes the proof.
\end{proof}
\begin{rem}\label{rem:admissibletuples}
	Let $A$ be the lower right $2\times 2$ part of $\mathrm{Ad}_w$. Then by \Cref{thm:fiberbundles} and Remark \ref{rem: Ti-compatible} any factorization $A=A_n\ldots A_1$ with $A_i$ being $T_i$-compatible can be completed to an admissible tuple. In fact completions correspond bijectively to choices of $B_1,\ldots,B_{m-2}$ subject to the condition that each of them annihilate a certain element of of $\Z^2$ defined by the respective $T_i$.
\end{rem}

It seems hard to classify all such solutions, but nonetheless we can find many for certain quasitoric bases, including all toric ones that are not $\C \PP^2$, and lots of Lie groups $G$.
		\begin{example}\label{ex:fiberbundles}
			Suppose that there is a subalgebra $\su(3)\subset \mfg$ for which 
			\begin{itemize}
				\item $T$ decomposes as a product of a codimension 2 subtorus $T'\subset T$ and
					the $2$-torus coming from $\su(3)$ such that the adjoint representation of $T'$
					fixes $\su(3)$. 
				\item there exists a Weyl group element $w\in W(G)$ that fixes $T'$ and acts on the maximal abelian subalgebra of $\su(3)$ as $e_1\mapsto -e_3$, $e_2\mapsto e_1$, where $e_1=2\pi i\diag(1,-1,0)$,
				$e_2=2\pi i\diag(0,1,-1)$ and $e_3=2\pi i\diag(1,0,-1)$ (in the usual identification of $W(\su(3))$ with $S_3$, this is the element $(23) \circ (12)$).
			\end{itemize}
One already obtains interesting examples by choosing $\mfg=\su(3)$ and $T'=\{e\}$.

			For some $b\in \Z$, we choose the basis $w_{m-1}=e_1-e_2$ and $w_m=-(b-1) e_1+be_2$ for the canonical lattice of the maximal torus in $\su(3)$ considered to be in $\mathfrak{g}$,
			and a basis $w_1,\hdots,w_{m-2}$ for the canonical lattice of $\mft'$. By the first
		of the two above assumptions, $w_1,\hdots,w_m$ is a basis of $\mathfrak{t}$.
			The matrix of the isomorphism $\Ad_w$
			with respect to the basis $w_{m-1}$ and $w_m$ is now
			\[
			\Ad_w=
			\begin{pmatrix}
				b & b-1 \\
				1 & 1
			\end{pmatrix}
			\begin{pmatrix}
				-1 & 1 \\
				-1 & 0
			\end{pmatrix}
			\begin{pmatrix}
				1 & -b+1 \\
				-1 & b
			\end{pmatrix}
			=\begin{pmatrix}
				-3b+1 & 3b^{2}-3b+1 \\
				-3 & 3b-2
			\end{pmatrix}.
			\]

			Let $\overline{B}$ be a graph of a $4$-dimensional effective quasitoric
			$T^2$-manifold where there are at least two edges labeled with the same
			weight\footnote{This is the case for all toric manifolds except $\C \PP^2$, since
			these can be obtained from repeated blow ups of a Hirzebruch surface, see
		\cite[Theorem 10.4.3]{MR2810322}}; without loss of generality, we may assume that these are the edges $(v_l,v_{l+1})$ and $(v_n,v_1)$ for some $2\leq l\leq n-2$.
			 We pull back the weights on $\overline{B}$ along a homomorphism $T\rightarrow
			 T/T'\cong T^2$, giving us the graph $B$ labeled with linear forms
			 $\alpha_1,\hdots,\alpha_n$ in $\mathfrak{t}^*$, such that $w_{m-1}$ is in
			 $\mft_n=\ker(\alpha_n)$ and $w_m$ is in $\mft_{n-1}=\ker(\alpha_{n-1})$. Then by
			 assumption, $\ker(\alpha_l)=\mft_l$ equals
			 $\mft_n$, and both are spanned by $w_1,\hdots, w_{m-1}$ while
			 $\mathfrak t_{n-1}$ is spanned by $w_1,\hdots, w_{m-2},w_m$.\\
			We set $A_k$ to be the identity for $k\neq n-1,n,l$, and need to choose $A_{l}$,
			$A_{n-1}$ and $A_n$ such that $\Ad_w=A_n\cdot A_{n-1}\cdot A_{l}$.
			Furthermore the $A_i$ need to satisfy the $T_i$-compatibility
			conditions which here is equivalent to $A_l,A_n$ being upper triangular with $1$ in
		the top left and $A_{n-1}$ being lower triangular with $1$ in the bottom right.\\
			We can therefore make the ansatz
		  \[
		  	A_{n}=\begin{pmatrix}
		  	1 & b \\
				0 & 1
		  	\end{pmatrix} 
		  \]
		  which leads to
			\[
				A_{n-1} \cdot A_{l} = A_{n}^{-1} \cdot \Ad_{w} = 
			\begin{pmatrix}
				1 & -b \\
				0 & 1
			\end{pmatrix}
			\begin{pmatrix}
				-3b+1 & 3b^{2}-3b+1 \\
				-3 & 3b-2
			\end{pmatrix}=
			\begin{pmatrix}
				1 & -b+1 \\
				-3 & 3b-2 			\end{pmatrix}
			\]
			Thus, we can choose
			\[A_{n-1}=\begin{pmatrix}
				1 & 0 \\
				-3 & 1
			\end{pmatrix}, \quad
			A_l=\begin{pmatrix}
				1 & -b+1 \\
				0 & 1
			\end{pmatrix}.
			\]
			By \Cref{thm:fiberbundles} and Remark \ref{rem:admissibletuples}, we obtain many admissible tuples in that setting.
		\end{example}
\bibliographystyle{plain}
\bibliography{gkm_flag_bundles12.0}
\end{document}